\documentclass[a4paper, twoside]{article}
\usepackage{a4}
\usepackage{amssymb}
\usepackage{amsmath}
\usepackage{upref}
\usepackage{graphics,epsfig,graphicx,multicol}
\usepackage{psfrag}
\usepackage[active]{srcltx}
\usepackage[dvips,colorlinks,citecolor=blue,linkcolor=blue]{hyperref}
\usepackage[dvipsnames]{color}
\usepackage[normalem]{ulem}
\allowdisplaybreaks[2] % To make it possible for displaybreaks
\newcount\minutes \newcount\hours
\hours=\time
\divide\hours 60
\minutes=\hours
\multiply\minutes -60
\advance\minutes \time
\newcommand{\klockan}{\the\hours:{\ifnum\minutes<10 0\fi}\the\minutes}
\newcommand{\tid}{\today\ \klockan}
\newcommand{\prtid}{\smash{\raise 10mm \hbox{\LaTeX ed \tid}}}
\renewcommand{\prtid}{}
\makeatletter
\def\sectionmark#1{} %\markboth{{\sectnr #1}}{{\sectnr #1}}} %Journal
\def\subsectionmark#1{}
\newcommand{\sectnr}{\ifnum \c@secnumdepth >\z@
                 \thesection.\hskip 1em\relax \fi}
\def\@evenhead{\footnotesize\rm\thepage\hfil\leftmark\hfil\llap{\prtid}}
\def\@oddhead{\footnotesize\rm\rlap{\prtid}\hfil\rightmark\hfil\thepage}
\def\tableofcontents{\section*{Contents} %\@mkboth{Contents}{Contents}} %Journal
 \@starttoc{toc}}
\makeatother
\makeatletter
\def\@biblabel#1{#1.}
\makeatother
\makeatletter
\let\Thebibliography=\thebibliography
\renewcommand{\thebibliography}[1]{\def\@mkboth##1##2{}\Thebibliography{#1}
\addcontentsline{toc}{section}{References}
\frenchspacing % Maybe not needed
\setlength{\@topsep}{0pt}% Delete if extra space before list
\setlength{\itemsep}{0pt}%
\setlength{\parskip}{0pt plus 2pt}%
}
\makeatother
\makeatletter
\def\mdots@{\mathinner.\nonscript\!.%
 \ifx\next,.\else\ifx\next;.\else\ifx\next..\else
 \nonscript\!\mathinner.\fi\fi\fi}
\let\ldots\mdots@
\let\cdots\mdots@
\let\dotso\mdots@
\let\dotsb\mdots@
\let\dotsm\mdots@
\let\dotsc\mdots@
\def\vdots{\vbox{\baselineskip2.8\p@ \lineskiplimit\z@
    \kern6\p@\hbox{.}\hbox{.}\hbox{.}\kern3\p@}}
\def\ddots{\mathinner{\mkern1mu\raise8.6\p@\vbox{\kern7\p@\hbox{.}}%
    \raise5.8\p@\hbox{.}\raise3\p@\hbox{.}\mkern1mu}}
\makeatother
\makeatletter
\let\Enumerate=\enumerate
\renewcommand{\enumerate}{\Enumerate%
\setlength{\@topsep}{0pt}% Delete if extra space before list
\setlength{\itemsep}{0pt}%
\setlength{\parskip}{0pt plus 1pt}%
\renewcommand{\theenumi}{\textup{(\alph{enumi})}}%
\renewcommand{\labelenumi}{\theenumi}%
}
\let\endEnumerate=\endenumerate
\renewcommand{\endenumerate}{\endEnumerate\unskip}
\newcounter{saveenumi}
\makeatother
\makeatletter
\def\@seccntformat#1{\csname the#1\endcsname.\quad}
\makeatother
\makeatletter
\long\def\@makecaption#1#2{%
  \vskip\abovecaptionskip
  \sbox\@tempboxa{ #1. #2}%
  \ifdim \wd\@tempboxa >\hsize
    #1. #2\par
  \else
    \global \@minipagefalse
    \hb@xt@\hsize{\hfil\box\@tempboxa\hfil}%
  \fi
  \vskip\belowcaptionskip}
\makeatother
\newcommand{\authortitle}[2]{\author{#1}\title{#2}\markboth{#1}{#2}}
\newcommand{\art}[6]{{\sc #1, \rm #2, \it #3\/ \bf #4 \rm (#5), \mbox{#6}.}}
\newcommand{\artnopt}[6]{{\sc #1, \rm #2, \it #3\/ \bf #4 \rm (#5), \mbox{#6}}}
\newcommand{\auth}[2]{{#2. #1}}
\def\idxauth{\auth}
\newcommand{\arttoappear}[3]{{\sc #1, \rm #2, to appear in \it #3}}
\newcommand{\artin}[3]{{\sc #1, \rm #2, in #3.}}
\newcommand{\book}[3]{{\sc #1, \it #2, \rm #3.}}
\newcommand{\AND}{{\rm and }}
\RequirePackage{amsthm}
\newtheoremstyle{descriptive}%
  {\topsep}   %{\medskipamount}          % Space above
  {\topsep}   %  {\medskipamount}          % Space below
  {\rmfamily} % Body font
  {}          % Indent
  {\bfseries} % Head font
  {.}         % Punctuation after thm head
  { }         % Space after thm head
  {}          % Thm head spec(?)
\newtheoremstyle{propositional}%
  {\topsep}   %  {\medskipamount}          % Space above
  {\topsep}   %  {\medskipamount}          % Space below
  {\itshape}  % Body font
  {}          % Indent
  {\bfseries} % Head font
  {.}         % Punctuation after thm head
  { }         % Space after thm head
  {}          % Thm head spec(?)
\theoremstyle{propositional}
\newtheorem{thm}{Theorem}[section]
\newtheorem{theo}[thm]{Theorem}   % Alternative
\newtheorem{prop}[thm]{Proposition}
\newtheorem{lem}[thm]{Lemma}
\newtheorem{cor}[thm]{Corollary}
\newtheorem{lemma}[thm]{Lemma}
\theoremstyle{descriptive}
\newtheorem{definition}[thm]{Definition}
\newtheorem{example}[thm]{Example}
\newtheorem{remark}[thm]{Remark}
\makeatletter
\renewenvironment{proof}[1][\proofname]{\par
  \pushQED{\qed}%
  \normalfont
  \trivlist
  \item[\hskip\labelsep
        \itshape
    #1\@addpunct{.}]\ignorespaces
}{%
  \popQED\endtrivlist\@endpefalse
}
\makeatother
{\catcode`p =12 \catcode`t =12 \gdef\eeaa#1pt{#1}}      % Get slantfactor
\def\accentadjtext#1{\setbox0\hbox{$#1$}\kern   % Convert it with height
                \expandafter\eeaa\the\fontdimen1\textfont1 \ht0 }
\def\accentadjscript#1{\setbox0\hbox{$#1$}\kern % Convert it with height
                \expandafter\eeaa\the\fontdimen1\scriptfont1 \ht0 }
\def\accentadjscriptscript#1{\setbox0\hbox{$#1$}\kern   % Convert it with height
                \expandafter\eeaa\the\fontdimen1\scriptscriptfont1 \ht0 }
\def\accentadjtextback#1{\setbox0\hbox{$#1$}\kern       % Convert it with height
                -\expandafter\eeaa\the\fontdimen1\textfont1 \ht0 }
\def\accentadjscriptback#1{\setbox0\hbox{$#1$}\kern     % Convert it with height
                -\expandafter\eeaa\the\fontdimen1\scriptfont1 \ht0 }
\def\accentadjscriptscriptback#1{\setbox0\hbox{$#1$}\kern % Convert it with height
                -\expandafter\eeaa\the\fontdimen1\scriptscriptfont1 \ht0 }
\def\itoverline#1{{\mathsurround0pt\mathchoice
        {\rlap{$\accentadjtext{\displaystyle #1}
                \accentadjtext{\vrule height1.593pt}
                \overline{\phantom{\displaystyle #1}
                \accentadjtextback{\displaystyle #1}}$}{#1}}
        {\rlap{$\accentadjtext{\textstyle #1}
                \accentadjtext{\vrule height1.593pt}
                \overline{\phantom{\textstyle #1}
                \accentadjtextback{\textstyle #1}}$}{#1}}
        {\rlap{$\accentadjscript{\scriptstyle #1}
                \accentadjscript{\vrule height1.593pt}
                \overline{\phantom{\scriptstyle #1}
                \accentadjscriptback{\scriptstyle #1}}$}{#1}}
        {\rlap{$\accentadjscriptscript{\scriptscriptstyle #1}
                \accentadjscriptscript{\vrule height1.593pt}
                \overline{\phantom{\scriptscriptstyle #1}
                \accentadjscriptscriptback{\scriptscriptstyle #1}}$}{#1}}}}
\def\itunderline#1{{\mathsurround0pt\mathchoice
        {\rlap{$\underline{\phantom{\displaystyle #1}
                \accentadjtextback{\displaystyle #1}}$}{#1}}
        {\rlap{$\underline{\phantom{\textstyle #1}
                \accentadjtextback{\textstyle #1}}$}{#1}}
        {\rlap{$\underline{\phantom{\scriptstyle #1}
                \accentadjscriptback{\scriptstyle #1}}$}{#1}}
        {\rlap{$\underline{\phantom{\scriptscriptstyle #1}
                \accentadjscriptscriptback{\scriptscriptstyle #1}}$}{#1}}}}
\newcommand{\limplus}{{\mathchoice{\raise.17ex\hbox{$\scriptstyle +$}}
        {\raise.17ex\hbox{$\scriptstyle +$}}
        {\raise.1ex\hbox{$\scriptscriptstyle +$}}
        {\scriptscriptstyle +}}}
\newcommand{\limminus}{{\mathchoice{\raise.17ex\hbox{$\scriptstyle -$}}
        {\raise.17ex\hbox{$\scriptstyle -$}}
        {\raise.1ex\hbox{$\scriptscriptstyle -$}}
        {\scriptscriptstyle -}}}
\newdimen\extrawidth
\def\iintlim#1#2{\setbox0\hbox{$\scriptstyle#1$}%
	\setbox1\hbox{$\scriptstyle#2$}%
	\extrawidth=\wd1 \advance\extrawidth-\wd0
	\ifdim\extrawidth<0pt \extrawidth=0pt\fi%
	\int_{#1\kern\extrawidth \kern .5em}^{#2\kern -\wd1} \kern -.5em%
}
\numberwithin{equation}{section}
\def\cprime{{\mathsurround0pt$'$}}
\newcommand{\imp}{\ensuremath{\mathchoice{\quad \Longrightarrow \quad}{\Rightarrow}
                {\Rightarrow}{\Rightarrow}} }
\newenvironment{ack}{\medskip{\it Acknowledgement.}}{}

\DeclareMathOperator{\diam}{diam}
\DeclareMathOperator{\Div}{div}

\newcommand{\vp}{\varphi}
\renewcommand{\phi}{\varphi}
\newcommand{\eps}{\varepsilon}
\newcommand{\lm}{\lambda}
\newcommand{\al}{\alpha}
\newcommand{\de}{\delta}
\newcommand{\R}{\mathbf{R}}
\newcommand{\Rn}{\mathbf{R}^n}
\newcommand{\Rno}{\mathbf{R}^{n+1}}
\newcommand{\parts}[2]{\frac{\partial {#1}}{\partial {#2}}}
\newcommand{\abs}[1]{\left| #1 \right|}
\newcommand{\defiref}[1]{Definition~\ref{#1}}
\newcommand{\Om}{\Omega}
\newcommand{\p}{{$p\mspace{1mu}$}}
\newcommand{\uP}{\itoverline{H}}
\newcommand{\uPa}{{\itoverline{H}\mspace{1mu}}^a}
\newcommand{\lP}{\itunderline{H}}
\newcommand{\uPind}[1]{\itoverline{H}_{#1}}
\newcommand{\UU}{\mathcal{U}}%
\newcommand{\LL}{\mathcal{L}}%
\newcommand{\bdy}{\partial}
\newcommand{\bdry}{\partial}
\newcommand{\bdyp}{\bdy_p}
\newcommand{\grad}{\nabla}
\newcommand{\setm}{\setminus}
\renewcommand{\emptyset}{\varnothing}
\newcommand{\alp}{\alpha}
\newcommand{\ga}{\gamma}
\newcommand{\wt}{\widetilde{w}}
\newcommand{\Kt}{\widetilde{K}}
\newcommand{\Gt}{\widetilde{G}}
\newcommand{\Wp}{W^{1,p}}
\newcommand{\Thetam}{\Theta_\limminus}
\newcommand{\efoft}{\biggl(\frac{|{\log(-t)}|^{p-2}-1}{p-2}\biggr)}
\newcommand{\efoftsing}{\biggl(\frac{|{\log(-t)}|^{2-p}-1}{2-p}\biggr)}
\begin{document}
%
% \authortitle sets headlines and also \author and \title, but
% the latter can be changed afterwards if needed for page 1
%
\authortitle{Anders Bj\"orn, Jana Bj\"orn, Ugo Gianazza
    and Mikko Parviainen}
{Boundary regularity for degenerate and singular parabolic equations}
\author{
Anders Bj\"orn \\
\it\small Department of Mathematics, Link\"opings universitet, \\
\it\small SE-581 83 Link\"oping, Sweden\/{\rm ;}
\it \small anders.bjorn@liu.se
\\
\\
Jana Bj\"orn \\
\it\small Department of Mathematics, Link\"opings universitet, \\
\it\small SE-581 83 Link\"oping, Sweden\/{\rm ;}
\it \small jana.bjorn@liu.se
\\
\\
Ugo Gianazza \\
\it\small Department of Mathematics ``F. Casorati'', Universit\`a di Pavia,\\
\it\small via Ferrata 1, IT-27100 Pavia, Italy\/{\rm ;}
\it\small gianazza@imati.cnr.it
\\
\\
Mikko Parviainen \\
\it\small Department of Mathematics and Statistics, University of Jyv\"askyl\"a,\\
\it\small P.O. Box 35\/  \textup{(}MaD\/\textup{)}, 
  FI-40014 Jyv\"askyl\"a, Finland\/{\rm ;}
\it\small mikko.j.parviainen@jyu.fi
}

\date{}
\maketitle

\noindent{\small
{\bf Abstract}.
We characterise regular boundary points of the parabolic \p-Laplacian
in terms of a family of barriers, both when $p>2$ and $1<p<2$. Due to the fact that
$p\not=2$, it turns out that one can multiply the \p-Laplace operator by a positive constant,
without affecting the regularity of a boundary point. By constructing suitable families 
of barriers, we give some simple geometric conditions that ensure the regularity of 
boundary points.
}

\bigskip
\noindent
{\small \emph{Key words and phrases}:
barrier family, cone condition, degenerate parabolic,
exterior ball condition,
\p-Laplacian, Perron's method, Petrovski\u\i\   criterion,
regular boundary point,
singular parabolic.
}


\section{Introduction}

A boundary point  is \emph{regular} {with respect} to a partial differential equation
if all solutions to the Dirichlet problem attain their continuous boundary values continuously
at that point.
The characterisation of regular boundary points for different partial differential equations has a very long history. Wiener gave a necessary and sufficient condition,
the \emph{Wiener criterion}, 
for the boundary regularity in the context of  the Laplace equation in his celebrated 1924 paper \cite{wiener}.
  Evans and Gariepy~\cite{evansg82} settled the question about the boundary regularity for the heat equation, but the boundary regularity for \p-parabolic type equations in terms of an explicit Wiener type criterion is a long standing open problem. 
   In this paper, we characterise regular boundary points for the \p-parabolic equation, $1<p<\infty$, 
\[
\begin{split}
\parts{u}{t}=\Div(\abs{\nabla u}^{p-2} \nabla u)
\end{split}
\]
in terms of a family of barrier functions (Theorem~\ref{thm:barrier-char}). 

The parabolic boundary regularity is a quite delicate question, and this is already apparent,
when dealing with the one-dimensional linear heat equation. Indeed, 
a boundary point can be regular for the equation
\[
  \frac{\partial u}{\partial t} = \frac{\partial^2 u}{\partial x^2},
\]
but irregular for its (still linear) cousin
\[
2\frac{\partial u}{\partial t} = \frac{\partial^2 u}{\partial x^2}.
\] 
This can be seen using
Petrovski\u\i's   criterion for the one-dimensional heat equation,
see \cite{Petro1,Petro2}. The interested reader can refer to the historical introduction presented
by Galaktionov~\cite{galaktionov}, in particular for the improvements with respect to the original result of \cite{Petro1}, that Petrovski\u\i\  proved in \cite{Petro2}.

Such a behaviour might raise a serious doubt on all the attempts for proving the 
regularity of boundary points for parabolic problems, let alone a parabolic Wiener's criterion, 
based on estimates with unspecified constants. 
However, we show in Theorem~\ref{thm:multiplied-eq}, perhaps surprisingly, 
that in the nonlinear case, the regularity of a boundary point is not affected by 
multiplicative constants on one side of the equation.  
We also derive several concrete characterisations of
the regularity of boundary points: by constructing an explicit family of barriers, we show that the exterior ball condition and several exterior cone conditions imply the regularity of a boundary point. 
Moreover, the barrier family characterisation implies that the regularity of a boundary point is of local nature and that the future does not affect the regularity of a point. 
Finally, we establish a Petrovski\u\i\  type regularity 
condition for the latest boundary point. 

Perron's method \cite{perron23}, also called the
Perron--Wiener--Brelot method, originally developed for  the Laplace equation, has become a fundamental tool in 
the study of Dirichlet boundary value problems for various  
elliptic and parabolic 
partial differential equations. 
The idea is to construct an upper  solution of the Dirichlet problem as an 
infimum of a certain upper class of supersolutions. A lower
solution is constructed similarly using a lower class of subsolutions,
and when the upper and lower solutions coincide we obtain a reasonable 
solution.
 In the Perron method, the boundary regularity is essentially a separate problem from 
the existence of a solution.  

Systematic use of barrier functions as a tool for studying boundary regularity 
seems to date back to the 1912 work of Lebesgue~\cite{lebesgue1912}. Later, Lebesgue~\cite{Lebesgue}
characterised regular boundary points in terms of barriers
for the linear Laplace equation.
In the elliptic setting, the extension of Perron's method and the method of barriers to the nonlinear \p-Laplacian
was initiated by Granlund, Lindqvist and Martio in \cite{GLM86}
and developed in a series of papers (see, for example, the accounts given in 
Heinonen--Kilpel\"ainen--Martio~\cite{HeKiMa} and  Bj\"orn--Bj\"orn~\cite{BBbook}).

Coming to the heat equation, in much of the existing literature, the subject has been addressed as an analog to the classical theory for the Laplace equation 
(see e.g.\ Doob~\cite{doob}) or as an example of abstract potential theory 
(see e.g.\ Bliedtner--Hansen~\cite{bliedtnerhansen}). 
However, the recent 
book by Watson \cite{watson} deals with heat potential theory as a subject on its own. 

The potential theory for \p-\emph{parabolic} type equations was initiated by Kilpel\"ainen and Lindqvist in \cite{KiLi96}. They established the parabolic Perron method, and also suggest a boundary regularity condition in terms of one barrier function. Even if the single barrier criterion has turned out to be 
problematic,\footnote{We thank P.\ Lindqvist for bringing this problem to our attention.} our paper owes a lot of inspiration and techniques to \cite{KiLi96}, as well as \cite{lindqvist95}.

The paper is  organised as follows: Section~\ref{S:Prelim}
is devoted to some preliminary material. In particular, we recall the different concepts of solutions -- weak supersolutions and superparabolic functions -- as well as the Perron method.
Section~\ref{S:Boundary} deals with the boundary regularity, the definition of a family of barriers
and the characterisation of a regular boundary point in terms of barriers, with some related properties.
 
The following sections are devoted to simple geometric criteria for the regularity of boundary points, which are all derived by constructing suitable families of barriers. In particular, the exterior ball condition of Section~\ref{sect-ball-all-p} states that if there is an exterior ball touching the domain at a point, then this point is a regular boundary point with two exceptions:
Consider the Dirichlet problem in a space time cylinder $\Om\times(0,T)$. The boundary values are given on the lateral boundary $\partial \Om\times(0,T)$ as well as on the initial boundary $\Om\times\{0\}$. However, the solution itself determines the values on $\Om\times \{T\}$. Thus, evidently, the point of contact should not be the south pole of the exterior ball. However, curiously 
also the `north pole' as a point of contact causes difficulties when using the natural family of barriers. By different means, one easily sees that for example $(x,t)\in \Om\times \{0\}$ is a regular point in the cylindrical case. 
However, if the initial boundary is for example  a half sphere (like at least roughly in a soda can), then it would be interesting to know if the north pole of the initial boundary is regular.

Section~\ref{S:cone} deals with cone conditions, 
while Section~\ref{S:petr} deals with
a Petrovski\u\i\ type condition for $p>2$.
Petrovski\u\i~\cite{Petro2} showed that $(0,0)$ is regular for the heat equation
with respect to 
\[
 \{(x,t): |x|<2(1+\eps)\sqrt{-t}\sqrt{\log |{\log(-t)}|} \text{ and } -1\le t<0\},
\]
if $\eps=0$, while it is
irregular if $\eps >0$.
In Section~\ref{S:petr} we obtain a similar result 
for $p>2$.
We have not been able to obtain a 
Petrovski\u\i\ type condition for $1<p<2$,
but in Section~\ref{S:final} we deduce a somewhat weaker result.
We end the paper by giving  
a list of some
open problems in Section~\ref{S:Open}.

\begin{ack}
A.\ B.\ and J.\ B.\ are supported by the Swedish Research Council, and M.\ P.\ by the Academy of Finland.
Part of this research was done during several visits:
of M.\ P.\ to Link\"opings universitet in 2007,
of A.\ B.\ to Universit\`a di Pavia in 2011,
of U.\ G.\ to University of Jyv\"askyl\"a in 2012,
and while all authors visited Institut Mittag-Leffler in 2013.
\end{ack}

%%%%%%%%%%%%%%%%%%%%%%%%%%%%%%%%%%%%
\section{Preliminaries}\label{S:Prelim}

Let $\Theta$ be a bounded nonempty open set in $\R^{n+1}$, $1<p<\infty$, and
$z=(x,t)\in \R^{n+1}$.
We consider the equation
\begin{equation} \label{eq:para}
\parts{u}{t}=\Delta_p u:=\Div(\abs{\nabla u}^{p-2} \nabla u),
\end{equation}
where the gradient $\grad u$ and the \emph{\p-Laplacian} $\Delta_p$
are taken with respect to $x$. 
This equation is \emph{degenerate} if $p>2$ and \emph{singular} if $1<p<2$.
For $p=2$ it is the usual heat equation.

Observe that if $u$ satisfies \eqref{eq:para},
and $a \in \R$, then $-u$ and $u+a$ also satisfy
\eqref{eq:para}, but (in general) $au$ does not.

In what follows, unless otherwise stated, $Q$ stands for a box
$Q=(a_1,b_1)\times\ldots\times(a_n,b_n)$
in $\Rn$, and the sets
$Q_T=Q\times (0,T)$ and $ Q_{t_1,t_2}=Q\times(t_1,t_2)$
are called \emph{space-time boxes}. Further,
$B(\xi_0,r)=\{z \in \R^{n+1} : |z-\xi_0| <r\}$
stands for the usual Euclidean ball in $\Rno$.

 Let $U$ be an
open set in $\Rn$.  The \emph{parabolic boundary} of a cylinder
$U_{t_1,t_2}:=U\times (t_1,t_2)\subset \R^{n+1}$ is
\[
\partial_p U_{t_1,t_2}=(\overline U\times \{t_1\})\cup(\partial U\times (t_1,t_2]).
\]
We define the \emph{parabolic boundary} of a finite union of open
cylinders $U^i_{t_1^i,t_2^i}$ as follows
\[
\partial_p \biggl(\bigcup_i  U^i_{t_1^i,t_2^i}\biggr) : =
\biggl(\bigcup_i \partial_p U^i_{t_1^i,t_2^i}\biggr) \setminus
\bigcup_i {U}^i_{t_1^i,t_2^i}.
\]
If $D'$ is a bounded open subset of
$D$ and the closure of $D'$ belongs to $D$, we 
write $D' \Subset
D$.
Note that the parabolic boundary is by definition compact.

Let $U$ be a bounded open set in $\Rn$.  As usual, $W^{1,p}(U)$
denotes the space of real-valued functions $f$ such that $f \in
L^p(U)$ and the distributional first partial derivatives $\partial
f /\partial x_i$, $i=1,2,\dots,n$, exist in $U$ and belong to
$L^p(U)$. We use the norm
\[
\|f\|_{W^{1,p}(U)} =\biggl(\int_U|f|^p\,dx + \int_U|\nabla f|^p\,dx\biggr)^{1/p}.
\]
The Sobolev space with zero boundary values, $W_0^{1,p}(U)$, is
the closure of $C_0^\infty(U)$ with respect to the Sobolev norm.

By the \emph{parabolic Sobolev space} $L^p(t_1,t_2;W^{1,p}(U))$,
with $t_1<t_2$, we mean the space of functions $u(x,t)$ such that the mapping $x
\mapsto u(x,t)$ belongs to $W^{1,p}(U)$ for almost every $t_1 < t <
t_2$ and the norm
\[
\biggl(\iintlim{t_1}{t_2}\int_U |u(x,t)|^{p} + |\nabla
u(x,t)|^{p}\,dx\,dt\biggr)^{1/p}
\]
is finite. The definition of the space
$L^p(t_1,t_2;W_{0}^{1,p}(U))$ is similar.
Analogously by the space $C(t_1,t_2;L^p(U))$,
with $t_1<t_2$, we mean the space of functions $u(x,t)$, such that the mapping
$t\mapsto\int_U|u(x,t)|^p\,dx$ is continuous in the time interval $[t_1,t_2]$.
(The gradient $\nabla$ and divergence $\Div$ are always taken with
respect to the $x$-variables in this paper.)
We can now introduce the notion of weak solution.

\begin{definition}
A function $u:\Theta \to [-\infty,\infty]$ is a  
\emph{weak solution} to equation \eqref{eq:para} 
if whenever $U_{t_1,t_2} \Subset \Theta$ is an open cylinder, 
we have 
$u \in C(t_1,t_2;L^2(U))\cap L^{p}(t_1,t_2;W^{1,p}(U))$, and  
 $u$ satisfies the integral equality
\[ 
\iintlim{t_1}{t_2}\int_{U} \abs{ \nabla u}^{p-2} \nabla u \cdot
\nabla\phi \, dx\,dt - \iintlim{t_1}{t_2}\int_{U} u
\frac{\partial\phi}{\partial t} \, dx\,dt   =  0
\quad \text{for all }\phi \in C_0^\infty(U_{t_1,t_2}).
\] 
Continuous weak
solutions are called  \emph{\p-parabolic functions}.

A function $u$
is a  \emph{weak supersolution} (\emph{subsolution}),
if whenever $U_{t_1,t_2} \Subset
\Theta$ we have $u \in
L^{p}(t_1,t_2;W^{1,p}(U))$, and the 
left-hand side 
above is nonnegative (nonpositive) for all
nonnegative $\phi \in C_0^\infty(U_{t_1,t_2})$.
\end{definition}

In the following, for simplicity, we will often 
omit \emph{weak}, when talking of weak (super)solutions.

Locally bounded solutions are locally H\"older continuous (see 
DiBenedetto~\cite{dibe-mono}, Chapters~III and IV). 
For $p>\frac{2n}{n+2}$
the notion of solution automatically yields 
local boundedness, whereas for $1<p\le\frac{2n}{n+2}$ explicit unbounded solutions are known, and in order to guarantee  
boundedness, an extra assumption on $u$ is needed 
(see the discussions in \cite{dibe-mono}, Chapter~V,
and DiBenedetto--Gianazza--Vespri~\cite{DBGV-mono}, Appendix~A). 
Although it plays no role in the following, it is worth mentioning that nonnegative solutions satisfy proper forms of Harnack inequalities (see \cite{DBGV-mono}, and also Kuusi~\cite{Kuusi08}), and locally bounded gradients of solutions are locally H\"older continuous (see \cite{dibe-mono}, Chapter~IX). 

\medskip

The so-called Barenblatt solution \cite{Bare52} to equation \eqref{eq:para}
\[
\mathcal B_p(x,t)=t^{-n/\lambda}\biggl(C-\frac{p-2}{p}\lambda^{1/(1-p)}
\biggl(\frac{|x|}{t^{1/\lambda}}\biggr)^{p/(p-1)}\biggr)_\limplus^{(p-1)/(p-2)},\quad\lambda=n(p-2)+p,
\]
is used in this paper.
Even though it was introduced in the context of degenerate equations for $p>2$,
it is well defined also for $p<2$, provided that $\lambda>0$,
which requires $\frac{2n}{n+1}<p<2$.

\begin{definition}\label{def:superparabolic}
A function $u:\Theta\rightarrow (-\infty,\infty]$
is \emph{\p-super\-parabolic} if
\begin{enumerate}
\renewcommand{\theenumi}{\textup{(\roman{enumi})}}%
\item $u$ is lower semicontinuous;
\item $u$ is finite in a dense subset of ${\Theta}$;
\item $u$ satisfies the following comparison principle on each space-time
  box $Q_{t_1,t_2}\Subset{\Theta}$: If $h$ is \p-parabolic in
  $Q_{t_1,t_2}$ and continuous on $\overline{Q}_{t_1,t_2}$, and if
  $h\leq u$ on $\partial_p Q_{t_1,t_2}$, then $h\leq u$ in the whole $Q_{t_1,t_2}$.
\end{enumerate}
\end{definition}
\noindent A \p-subparabolic function $v:\Theta\rightarrow [-\infty,\infty)$ 
is defined analogously,
except that $v$ 
is upper semicontinuous and the inequalities are reversed, i.e.\
 we require that if  $h\ge v$ on $\partial_p Q_{t_1,t_2}$, then $h\ge v$ in the whole $Q_{t_1,t_2}$.
Equivalently, $v$ is \p-subparabolic if $-v$ is \p-superparabolic.

It is also worth mentioning that the definition of \p-superparabolic functions given 
here is equivalent to
the modern definition of viscosity supersolutions as shown 
in Juutinen--Lindqvist--Manfredi~\cite{juutinenlm01}.

\begin{thm}\label{thm-cont-exist}
Consider a parabolic cylinder $G_T=G\times(0,T)$, where $G$ is a bounded domain in $\Rn$ with a Lipschitz boundary $\partial G$, and let 
$h\in C(\partial_p G_T)$. 
Then there is a unique \p-parabolic function 
$u \in C(\itoverline{G}_T)$ 
that is continuous in $G_T$ and takes the boundary values $u=h$ on the parabolic boundary 
$\partial_p G_T$. Moreover, if $h$ belongs to 
$C(t_1,t_2;L^2(G))\cap L^{p}(t_1,t_2;W^{1,p}(G))$, then so does $u$.
\end{thm}

We will need this result explicitly
when proving the 
parabolic comparison principle in
Theorem~\ref{thm:sekareuna},
but we also rely on it implicitly since it is used (at least for
boxes) when obtaining some results we quote from, e.g.,
Kilpel\"ainen--Lindqvist~\cite{KiLi96}.

Although it is well-known in the literature,
to our knowledge a full proof which
covers the whole range $1<p<\infty$ is given only in
Ivert~\cite{ivert}, Theorem~3.2, which in turn relies on Fontes~\cite{Fontes}.

In the proof of the `pasting lemma' (Lemma~\ref{lem-pasting}), we need the parabolic comparison principle which is essentially stated in \cite{KiLi96}, Lemma~4.3, and \cite{KoKuPa10}, Theorem~4.8. However, as a full proof seems to be difficult to find in the literature, we have chosen to write it down in details. In particular, there is a subtle technical issue when proceeding in time: At a first approach, one might be tempted to consider the Euclidean boundary of 
the whole finite union of space-time boxes at once, construct a suitable \p-parabolic comparison function $h$, and then compare $v$ as well as $u$ with $h$.
However, such a comparison function, \emph{continuous} on the whole  
Euclidean closure does not necessarily exist. Consider for example 
\[
\Xi = ((0,3)\times(0,1)) \cup ((1,2)\times(0,2)) \subset\R^{1+1},
\]
with the boundary data $\psi=0$ on $\bdry_p((0,3)\times(0,1))$ and $\psi=1$
on $\{1,2\}\times[1,2]$.
The function $\psi$ is continuous on $\bdry_p\Xi$ 
(and easily extends to a continuous function on the whole of $\bdry\Xi$)
but any \p-parabolic function
in $(0,3)\times(0,1)$ with $\psi$ as boundary values must be identically zero in 
$(0,3)\times(0,1)$, and thus cannot continuously attain the boundary value 1
at the corner points $(1,1)$ and $(2,1)$.

\begin{theo} \label{thm:sekareuna}
\textup{(Parabolic comparison principle)}
Let $\Theta$ be an open bounded set in $\R^{n+1}$.
Suppose that $u$ is \p-super\-parabolic and $v$ is \p-sub\-parabolic
in $\Theta$.
Let $T \in \R$ 
and assume that
 \begin{equation}  \label{eqn:liminf_vertailu}
  \infty \ne    \limsup_{\Theta \ni (y,s)\rightarrow (x,t)} v(y,s)\leq
   \liminf_{\Theta \ni (y,s)\rightarrow (x,t)} u(y,s) \ne -\infty
 \end{equation}
for all
$(x,t) \in\{(x,t) \in \partial\Theta : t< T\}$.
Then $v\leq u$ in $\{(x,t) \in \Theta : t<T\}$.
\end{theo}

For the proof we will rely on the following lemma.

\begin{lem} \label{lem-KKP}
Let $U=\bigcup_{i} Q^i$ be a finite union of boxes
and let $U_T:=U \times (0,T)$ be the corresponding cylinder.
Also let $u$ be a \p-superparabolic function in a neighbourhood
of $\overline{U}_T$, and $h \in C(\overline{U}_T)$ be a function
which is \p-parabolic in $U_T$ and such that $h \le u$
on $\bdy_p U_T$.
Then $h \le u$ in $U_T$.
\end{lem}

This lemma was stated and proved for the supercritical case 
$p>\frac{2n}{n+2}$ in 
Korte--Kuusi--Parviainen~\cite{KoKuPa10}, Lemma~4.1.
A careful check of the proof reveals that the requirement
$p>\frac{2n}{n+2}$ is used to ensure that
$v-u \in L^2$ in  the proof
of the comparison principle between sub- and supersolutions in
Lemma~3.5 in \cite{KoKuPa10}, which the proof 
of Lemma~4.1 relies on
through the proof 
of Lemma~3.6 in \cite{KoKuPa10}.
To also cover the range $1<p \le \frac{2n}{n+2}$
one can proceed as follows:
First assume that $u$ is bounded (that $h$ above is bounded
is automatic). Then the $L^2$ integrability is immediate for
$u$ and $h$. 
The proof of Lemma~4.1 in \cite{KoKuPa10} is
also relying on an existence result for obstacle problems, which is
Theorem~2.8 therein and whose proof can be found in 
Korte--Kuusi--Siljander~\cite[{Theorem~3.1}]{KoKuSi09} 
or in Lindqvist--Parviainen~\cite[{Theorem~3.2}]{lindqvistp12}.
These existence theorems in turn rely on Lemma~3.5 in \cite{KoKuPa10}
(but only applied with bounded functions
so that the $L^2$ integrability is automatic) and on a convergence result
for supersolutions (Theorem~5.3 in \cite{KoKuPa10}),
whose proof also applies in the subcritical case.
Thus, we have obtained Lemma~\ref{lem-KKP} above with $u$ \emph{bounded} for
all $1<p< \infty$.

Finally, to cover also the unbounded case when $p \le \frac{2n}{n+2}$,
let $m=\sup_{\overline{U}_T} h$ and $v=\min\{u,m\}$.
That $v$ is \p-superparabolic is immediate from the definition.
We can then apply  Lemma~\ref{lem-KKP} with  the bounded
functions $v$ and $h$, which yields that $h \le v$ in $U_T$.
Since $v \le u$, this concludes the proof also for the unbounded case.

\begin{proof}[Proof of Theorem~\ref{thm:sekareuna}]
Let $\eps>0$ and
\[
E= \{(x,t)\in\Theta: t\le T-\eps \text{ and } v(x,t)> u(x,t)+\eps\}.
\]
By \eqref{eqn:liminf_vertailu}, together with
the  compactness of $\{(x,t)\in\bdry\Theta: t\le T-\eps\}$, we conclude that
$\itoverline{E}$
is a compact subset of $\Theta$.
Assume that $E\ne \emptyset$, and
let
\[
T_0=\inf \{t : (x,t) \in E\} = \min \{t : (x,t) \in \itoverline{E}\}.
\]

Since $\itoverline{E}$ is compact, we can
find finitely many space-time boxes
$Q^i_{t^i_1,t^i_2}\Subset\Theta$ such that
$\bigcup_{i=1}^N Q^i_{t^i_1,t^i_2}\supset \itoverline{E}$,
where $t^i_1 \ne T_0$ and  $T_0 < t^i_2<T$, $i=1,2,\ldots, N$.
By changing the cover, we may assume that
$S:=\{t^i_j: j=1,2 \text{ and } i=1,2,\ldots,N\}$ only contains one value, say $\sigma$,  less than $T_0$.
Let $m$ be the number of boxes $Q^i_{t^i_1,t^i_2}$ with $t_1^i=\sigma$, and
assume that these are ordered first.
Let $\Xi:= \bigcup_{i=1}^m Q^i_{\sigma,\tau} \supset \{(x,t)  \in \itoverline{E} : t < \tau\}$,
where $\tau=\min\{t \in S : t > T_0\} > T_0$.

In particular, the parabolic boundary
$\bdry_p\Xi\subset\Theta\setm E$, and hence $v \le u+\eps$ on $\bdry_p\Xi$.
Thus there exists a continuous
function $\psi$ on $\bdry_p \Xi$ such that $v\le\psi\le u+\eps$.
By Theorem~\ref{thm-cont-exist},
we can find
a function $h \in C(\overline{\Xi})$ 
which is \p-parabolic in $\Xi$ 
and continuously attains its boundary values $h=\psi$
on $\bdry_p \Xi$.
Lemma~\ref{lem-KKP} applied in
$\Xi$ to $u+\eps$ and $h$, and to  $-v$ and $-h$, shows that
$v\le h \le u+\eps$ in $\Xi$.

Thus $\Xi \cap E = \emptyset$, and so $T_0 \ge \tau$, a contradiction.
Hence $E$ must be empty, and letting $\eps \to 0$ concludes the proof.
\end{proof}

A direct consequence of Theorem~\ref{thm:sekareuna}
is the following comparison principle,
which can be
considered as a sort of \emph{elliptic} version of the comparison principle,
since it does not acknowledge the presence of the parabolic boundary.
This elliptic comparison principle
is in fact equivalent to the fundamental inequality
$\lP f \le \uP f$ between lower and upper Perron solutions 
(see Definition~\ref{def-Perron} below).

\begin{theo} \label{thm-comp-princ}
\textup{(Elliptic-type comparison principle)}
Suppose that
 $u$ is \p-super\-pa\-ra\-bol\-ic and $v$ is \p-sub\-parabolic in $\Theta$. If
 \[ 
   \infty \ne    \limsup_{\Theta \ni (y,s)\rightarrow (x,t)} v(y,s)\leq
    \liminf_{\Theta \ni (y,s)\rightarrow (x,t)} u(y,s) \ne -\infty
  \] 
for all $(x,t) \in \partial\Theta$,
then $v\leq u$ in $\Theta$.
\end{theo}

The connection between \p-superparabolic functions and weak supersolutions 
is a delicate issue, 
see for example Kinnunen--Lindqvist~\cite{KinLin05}, \cite{KinLin06}, 
Kuusi~\cite{Kuusi09} and the survey in Lindqvist~\cite{Lin09}.
However, to conclude that our \emph{continuous} barriers below are \p-superparabolic for all $p >1$,  we only need to check that they are weak supersolutions and then use the following comparison principle for weak (sub/super)solutions, 
see Lemma~3.1 of 
Kilpel\"ainen--Lindqvist~\cite{KiLi96} and
Lemma~3.5 of Korte--Kuusi--Parviainen~\cite{KoKuPa10}.

\begin{lem} \label{lem-comp}
Suppose that $u$ is a weak supersolution and $v$ is 
a weak subsolution to \eqref{eq:para}
in a space-time cylinder $U_{t_1,t_2}$, where $U\subset\Rn$ is an open set. If $u$ and $-v$ are lower semicontinuous on $\overline{U}_{t_1,t_2}$ and $v\le u$ on the parabolic boundary $\partial_p U_{t_1,t_2}$, then $v\le u$ a.e.\ in $U_{t_1,t_2}$.
\end{lem}

\medskip

Let us now come to Perron's method for \eqref{eq:para}.
For us it will be enough to consider Perron solutions for
bounded functions, so for simplicity we restrict ourselves
to this case throughout this paper.

\begin{definition}   \label{def-Perron}
Given a bounded function $f \colon \bdy \Theta \to \R$,
let the upper class $\UU_f$ be the set of all
\p-superparabolic  functions $u$ on $\Theta$ which are
bounded below and such that
\begin{equation} \label{Uf-def}
    \liminf_{\Theta \ni \eta \to \xi} u(\eta) \ge f(\xi) \quad \text{for all }
    \xi \in \bdy \Theta.
\end{equation}
Define the \emph{upper Perron solution} of $f$  by
\[
    \uP f (\xi) = \inf_{u \in \UU_f}  u(\xi), \quad \xi \in \Theta.
\]
Similarly, let the lower class 
$\LL_f$ be the set of all
\p-subparabolic functions $u$ on $\Theta$ which are
bounded above and such that
\[
\limsup_{\Theta \ni \eta \to \xi} u(\eta) \le f(\xi) \quad \text{for all }
\xi \in \bdy \Theta,
\]
and define the \emph{lower Perron solution}  of $f$ by
\[
    \lP f (\xi) = \sup_{u \in \LL_f}  u(\xi), \quad \xi \in \Theta.
\]
\end{definition}

Note that we have an elliptic-type boundary condition
on the full boundary, not just a condition on
the possibly smaller parabolic boundary, whenever it is defined.

It follows from the elliptic-type comparison principle 
in Theorem~\ref{thm-comp-princ}
that $v\le u$ whenever $u\in \UU_f$ and $v\in \LL_f$. 
Hence $\lP f\le \uP f$.
Moreover, Kilpel\"ainen--Lindqvist~\cite{KiLi96}, Theorem~5.1,
proved that both $\lP f$ and $\uP f$
are \p-parabolic. 

The following lemma is useful when constructing new \p-superparabolic
functions.

\begin{lem} \label{lem-pasting}
\textup{(Pasting lemma)}
Let $G \subset \Theta$ be open. 
Also let $u$ and $v$ be \p-super\-pa\-ra\-bo\-lic in $\Theta$ and $G$,
respectively,
and let
\[
    w=\begin{cases}
     \min\{u,v\} & \text{in } G, \\
     u & \text{in } \Theta \setm G. \\
    \end{cases}
\] 
If $w$ is lower semicontinuous, then $w$ is \p-superparabolic in $\Theta$.
\end{lem}

\begin{proof}
Since $-\infty < w \le u$, $w$ is finite in a dense subset of $\Theta$,
and we only have to obtain the comparison principle.
Therefore, let $Q_{t_1,t_2} \Subset \Theta$ 
be a space-time box,
and $h \in C(\itoverline{Q}_{t_1,t_2})$ be \p-parabolic in $Q_{t_1,t_2}$ and such
that $h \le w$ on $\bdyp Q_{t_1,t_2}$.
Since $h \le u$ on $\bdyp Q_{t_1,t_2}$ and $u$ is \p-superparabolic, 
we directly have that $h \le u$ in $Q_{t_1,t_2}$.

Next let $\Gt=Q_{t_1,t_2} \cap G$ and $(x,t) \in \{(x,t) \in \partial\Gt : t< t_2\}$.
If $(x,t) \in G$, then $(x,t) \in \bdyp Q_{t_1,t_2}$ and thus
by the lower semicontinuity  of $v$,
\[
    \liminf_{\Gt \ni (y,s) \to (x,t)} v(y,s) \ge v(x,t) \ge h(x,t).
\]
On the other hand, if $(x,t) \notin G$, then,
by the lower semicontinuity of $w$, 
\[
    \liminf_{\Gt \ni (y,s) \to (x,t)} v(y,s) \ge w(x,t) = u(x,t) \ge h(x,t).
\]
Hence, 
the parabolic comparison principle in
Theorem~\ref{thm:sekareuna} shows that $h \le v$ in $\Gt$,
and thus $h \le w$ in $Q_{t_1,t_2}$.
\end{proof}

%%%%%%%%%%%%%%%%%%%%%%%%%%%%%%%%%%%%%%%%%%%%
\section{Boundary regularity}\label{S:Boundary}

\begin{definition}
\label{def:regular}
A boundary point $\xi_0\in \partial \Theta$ is 
\emph{regular} with respect to $\Theta$, if
\[
        \lim_{\Theta \ni \xi \to \xi_0} \uP f(\xi)=f(\xi_0)
\]
whenever $f: \partial \Theta \to \R$ is continuous.
Here $\uP f$ denotes the upper Perron solution of $f$.
\end{definition}

Observe that since $\lP f = - \uP (-f)$, regularity can equivalently
be formulated using lower Perron solutions. In the following we will
omit the explicit reference to $\Theta$, whenever no confusion may arise.

Our aim is next to characterise regular boundary points
using families of barriers. 
Such a characterisation serves two purposes: to give a criterion for regularity, 
and to deduce various consequences of regularity.
For the former, one would like to have as weak a condition as possible,
whereas for the latter a stronger condition is often useful.
Therefore, we introduce two conditions, which turn out to be equivalent.

\begin{definition}
\label{def-barrier} 
Let $\xi_0\in \partial \Theta$. A family of functions 
$w_j: \Theta\to (0,\infty]$, $j=1,2,\ldots$, 
 is a \emph{barrier family} in $\Theta$ at the point  $\xi_0$ if
for each $j$, 
\begin{enumerate}
 \item\label{cond-first}  $w_j$ is a  positive \p-superparabolic function in $\Theta$;
 \item\label{cond-third} $\lim_{{\Theta \ni} \zeta\to \xi_0} w_j(\zeta)=0$;
 \item\label{cond-second-weak} 
for each $k=1,2,\ldots$, there is a $j$ such that 
\[
   \liminf_{\Theta \ni \zeta\to\xi} w_j(\zeta) \ge k
   \quad \text{for all } \xi \in \bdy \Theta 
   \text{ with }|\xi-\xi_0| \ge 1/k.
\]
\setcounter{saveenumi}{\value{enumi}}
\end{enumerate}
\bigskip
%
% \bigskip needed since enumerate ends with a display.

We also say that the family $w_j$ 
is a  \emph{strong barrier family} in $\Theta$ at the point  $\xi_0$ if,
in addition  the following conditions hold:
\begin{enumerate}
\setcounter{enumi}{\value{saveenumi}}
\item\label{cond-cont}  $w_j$ is continuous in $\Theta$;
\item\label{cond-second-strong} 
there is a nonnegative function $d \in C(\overline{\Theta})$, with
$d(z)=0$ if and only if $z=\xi_0$,
such that 
for each $k=1,2,\ldots$, there is a $j=j(k)$ such that 
$w_j \ge k d$ in $\Theta$.
\end{enumerate}
\end{definition}

Note that in \ref{cond-second-weak} the conditions
on $w_j$ are only at $\bdy \Theta$, while in \ref{cond-second-strong}
there is a requirement on $w_j$ in all of $\Theta$.
The latter will be important when proving several of the consequences
of the barrier characterisation that we derive later
in this section.
Note also that \ref{cond-second-strong} $\imp$ \ref{cond-second-weak}.

In classical potential theory,  a \emph{barrier} is a superharmonic (when dealing with the 
Laplace equation) or superparabolic (when dealing with the heat equation) 
function $w$ 
such that 
\[ 
\lim_{\zeta\to \xi_0} w(\zeta)=0 
   \quad \text{and} \quad
  \liminf_{\zeta\to \xi} w(\zeta)>0
   \text{ for } \xi \in \bdy \Theta \setm \{\xi_0\}.
\]
Existence of such a \emph{single} barrier implies the regularity of a boundary point in the classical case, since  one can scale and lift the barriers. However, this is not the case with the parabolic \p-Laplacian, since the equation is not homogeneous with respect to $u$: this reflects in that a scaled weak supersolution is not necessarily a weak supersolution. We think that it is precisely this lack of homogeneity, which \emph{forces} the use of a whole family of barriers, instead of just simply one,
but we do not know if a family is really required.

We are now ready to characterise regularity
in terms of the existence of a barrier family (in our sense).
At the same time,  we show that the existence
of a strong barrier family is equivalent.
Note that we do \emph{not} show that every barrier family is a strong
barrier family, only that if there exists a barrier family then there 
also exists a strong barrier family.

\begin{thm} \label{thm:barrier-char}
Let $\xi_0 \in \bdy \Theta$. Then the following are equivalent:
\begin{enumerate}
\renewcommand{\theenumi}{\textup{(\arabic{enumi})}}%
\item \label{i-reg}
$\xi_0$ is regular\/\textup{;}
\item \label{i-weak}
there is a barrier family at $\xi_0$\/\textup{;}
\item \label{i-strong}
there is a strong barrier family at $\xi_0$.
\end{enumerate}
\end{thm}

\begin{proof}
\ref{i-weak} \imp \ref{i-reg}
First we show that if there is a 
barrier family at $\xi_0\in \partial \Theta$, then $\xi_0$ is a regular boundary point.
 Since $f$ in Definition~\ref{def:regular} is continuous, for each
$\eps>0$ there exists a constant  $\delta>0$ such that if
$\abs{\xi-\xi_0}<\delta$, $\xi \in \partial \Theta$, then $\abs{f(\xi)-f(\xi_0)}<\eps$. 
We can therefore choose
$j\geq 1$ large enough so that
\[
\liminf_{\Theta \ni \zeta\to\xi}
w_j (\zeta)+\eps+f(\xi_0)>f(\xi) 
\quad \text{for all } \xi \in \bdy \Theta.
\]
Thus $w_j+\eps+f(\xi_0)$  belongs
to the upper class $\mathcal U_f$,
and hence
\[
\limsup_{\Theta \ni \zeta\to\xi_0} \uP f(\zeta) 
\le \lim_{\Theta \ni \zeta\to\xi_0} w_j(\zeta)+\eps+f(\xi_0)=\eps+f(\xi_0).
\]
Since $-w_j-\eps+f(\xi_0)$ is in the lower class (multiplication
by $-1$ is allowed) if $j$ is large enough, 
we similarly obtain that 
\[
\liminf _{\Theta \ni \zeta\to\xi_0} \uP f(\zeta)
\ge 
\liminf _{\Theta \ni \zeta\to\xi_0} \lP f(\zeta) 
\ge -\eps+f(\xi_0).
\]
Letting $\eps \to 0$ shows that $\xi_0$ is regular.

\ref{i-reg} \imp \ref{i-strong}.
Next we prove that if $\xi\in\partial \Theta$ is regular, 
then there exists a strong barrier family at $\xi_0$. 

Without loss of generality we may assume that $\xi_0$ is the origin. 
For $(x,t)\in \R^{n+1}$ we define
\[
d(x,t)=\frac{p-1}{p}\abs{x}^{p/(p-1)}+\frac{n}{2\diam\Theta} t^2 
\]
and
\begin{equation} \label{eq-barrier-const}
\psi_j(x,t) 
= j \frac{p-1}{p}\abs{x}^{p/(p-1)} + j^{p-1}\frac{n}{2\diam\Theta} t^2
\ge \min\{j,j^{p-1}\} d(x,t).
\end{equation}
A straightforward computation shows that in $\Theta$,
\[
\partial_t{\psi_j}-\Delta_p \psi_j
= j^{p-1}\frac{nt}{\diam\Theta} - j^{p-1}n\leq 0,
\]
i.e.\ $\psi_j$ is \p-subparabolic. 
Setting $w_j:=\lP \psi_j$ 
gives us a strong barrier family at $\xi_0$. 
Indeed, \ref{cond-second-strong} in
\defiref{def-barrier} follows from the definition of the lower Perron solution,
as it yields $w_j \ge \psi_j$,
and \ref{cond-third}
from the fact that $\xi_0$ is regular. 

\ref{i-strong} \imp \ref{i-weak}
This is trivial.
\end{proof}

The first consequence of the barrier characterisation is the following
\emph{restriction} theorem.

\begin{prop} \label{prop-restrict}
Let $\xi_0 \in \bdy \Theta$ and let $G \subset \Theta$ be open and such
that $\xi_0 \in \bdy G$.
If $\xi_0$ is regular with respect to $\Theta$,
then $\xi_0$ is regular with respect to $G$.
\end{prop}

\begin{proof}
By Theorem~\ref{thm:barrier-char}, there is a 
strong barrier family $\{w_j\}_{j=1}^\infty$ in $\Theta$ at $\xi_0$. 
Let $d$ be as given in condition~\ref{cond-second-strong}.
Let also $w_j'=w_j|_G$, $j\ge 1$, and $d'=d|_G$.
Then $\{w_j'\}_{j=1}^\infty$ is a strong barrier family in $G$ at $\xi_0$, 
and thus Theorem~\ref{thm:barrier-char} implies that $\xi_0$ is a regular boundary point with respect to $G$.
\end{proof}

Another consequence of the barrier characterisation is that
regularity is a \emph{local} property.

\begin{prop} \label{prop-local}
Let $\xi_0 \in \bdy \Theta$ 
and $B$ be a ball containing $\xi_0$.
Then
$\xi_0$ is regular with respect to $\Theta$
if and only if $\xi_0$ is regular with respect to $B \cap\Theta$.
\end{prop}

\begin{proof}
Proposition~\ref{prop-restrict} immediately implies that if $ \xi_0$ is regular 
with respect to $\Theta$, then it is also regular with respect to $B\cap \Theta$.

Next we show that if $\xi_0$ is regular with respect to $B \cap\Theta$, 
then it is regular 
with respect to $\Theta$.
By Theorem~\ref{thm:barrier-char}, there is a strong barrier family 
$\{w_j\}_{j=1}^\infty$ in $B \cap\Theta$, and a nonnegative continuous function
$d$ associated with it.
Let $m=\inf_{\overline{\Theta} \cap\partial B}d>0$,
\begin{equation*}
d'=
\begin{cases}
\min\{d,m\}& \text{in }B \cap \overline{\Theta}\\
m&\text{in } \Rno\setm B
\end{cases}
\quad \text{and} \quad
w'_{k}=
\begin{cases}
\min\{w_{j(k)},km\}&\text{in } B\cap \Theta\\
km&\text{in }\Theta\setminus B,
\end{cases}
\end{equation*}
where $j(k)$ is as in Definition~\ref{def-barrier}\,\ref{cond-second-strong},
$k=1,2,\ldots$.

Then $w_{k}'$ is continuous in $\Theta$, satisfies $w_{k}'\ge kd'$ in $\Theta$, 
and by
the pasting lemma (Lemma~\ref{lem-pasting}), 
$w_{k}'$ is  \p-superparabolic  in $\Theta$.
Hence $\{w_{k}'\}_{k=1}^\infty$is the
desired strong barrier family in $\Theta$ at $\xi_0$, 
and this implies that $\xi_0$ is regular with respect to $\Theta$.
\end{proof}

Next we state one of our main results.
As we have already remarked, if $u$ is a (super)solution to the \p-parabolic 
equation and $\ga  \ge 0$,
then in general $\ga u$ is not a (super)solution to the same equation,
except when $p=2$. Instead, $\ga u$ is a solution to a multiplied
\p-parabolic equation, namely to
\[
a\frac{\partial u}{\partial t}
=\Delta_p u
,
\quad a=\gamma^{p-2},
\]
as is apparent by straightforward calculations. 
This fact makes it possible to show that when $p \ne 2$, the regular points
are the same for all multiplied \p-parabolic equations. This is quite surprising, 
because a similar statement is known to be false for the heat equation
as a direct consequence of Petrovski\u\i's criterion discussed in the introduction.

On the other hand, for $p=2$ it is enough to have one barrier to get regularity,
since any positive multiple of a barrier is still a barrier.
So the gist of the argument is the following: for $p=2$ one barrier is enough, but multiplied equations have different
regular points, whereas for $p\ne2$ a family of barriers is required 
(we believe), but a regular point is such for all multiplied equations.

\begin{theo} \label{thm:multiplied-eq}
Let $\xi_0 \in \bdy \Theta$ and $a>0$.
If $p \ne 2$,
then $\xi_0$ is regular if and only if it is
regular with respect to the multiplied \p-parabolic  equation
\begin{equation} \label{eq-mult-eq}
      a\parts{u}{t}=\Delta_p u. 
\end{equation}
\end{theo}

\begin{proof}
Let $w$ be a weak supersolution
to the \p-parabolic equation and let $\wt = a^{1/(p-2)}w$.
Then
\[
a\,\partial_t{\wt}-
\Delta_p \wt
=a^{1+1/(p-2)}\partial_t{w}-a^{1+1/(p-2)}
\Delta_p w
\geq 0,
\]
and thus $\wt$  is a weak supersolution to the multiplied
\p-parabolic  equation, and vice versa.
The same equivalence obviously holds also for \p-superparabolic functions.

It follows directly that $u \in \UU_f$ if and only if
$a^{1/(p-2)} u \in \UU_{a^{1/(p-2)}f}^a$, where $\UU_f^a$ is the upper class
defining the upper Perron solution with respect to \eqref{eq-mult-eq}.
The equivalence of regularity of $\xi_0$ with respect to
\eqref{eq:para} and with respect to
\eqref{eq-mult-eq} now follows directly from the definition.
\end{proof}

It is noteworthy that Theorem~\ref{thm:multiplied-eq} holds both for $p>2$ and
$1<p<2$.
The next corollary immediately follows from the proof of the previous result.
\begin{cor}
Let $f\in C(\partial \Theta)$. Then 
\begin{equation} \label{eq-uPa}
      \uPa (a^{1/(p-2)}f) = a^{1/(p-2)} \uP f,
\end{equation}
where $\uPa$ denotes the upper Perron solution with respect
to \eqref{eq-mult-eq}.
\end{cor}

\begin{lem} \label{lem-component}
Assume that\/ $\Theta_1,\ldots,\Theta_m$ are pairwise disjoint bounded open sets
in $\Rno$ with $\xi_0 \in \bdy \Theta_j$, $j=1,\ldots,m$.
Then $\xi_0$ is regular with respect to $\Theta=\bigcup_{j=1}^m \Theta_j$
if and only if it is regular with respect to each $\Theta_j$, $j=1,\ldots,m$.
\end{lem}
\begin{proof}
The necessity follows from Proposition~\ref{prop-restrict}.
As for the sufficiency, let $f \in C(\bdy \Theta)$.
Then $(\uP f)|_{\Theta_j}= \uPind{\Theta_j} f|_{\bdy \Theta_j}$ and thus
\[
     \lim_{\Theta_j \ni \xi \to \xi_0} \uP f(\xi)
     = \lim_{\Theta_j \ni \xi \to \xi_0} \uPind{\Theta_j} f|_{\bdy \Theta_j}(\xi)
       = f(\xi_0)
\]
for $j=1,\ldots,m$.
It follows that
\[
     \lim_{\Theta \ni \xi \to \xi_0} \uP f(\xi)
       = f(\xi_0),
\]
and hence $\xi_0$ is regular with respect to $\Theta$.
\end{proof}

Next we let $\Theta=G_T=G\times(0,T)$, where $G\subset\Rn$ is an open set, and recall 
two results from Kilpel\"ainen--Lindqvist~\cite{KiLi96}. 
As we want to use our barrier family characterisation, we sketch the proofs 
in this context for the convenience of the reader.

\begin{theo}  \label{thm-ell-reg-iff-par-reg}
Let $x_0\in \partial G $  and $0<t_0\le T$. 
Then the boundary point $\xi_0= (x_0,t_0)$ 
is regular with respect to $G_T$, in the sense of Definition~\ref{def:regular}, 
if and only if $x_0$ is regular for \p-harmonic functions
with respect to $G$. 
\end{theo}

A \emph{\p-harmonic function} is a continuous weak solution to the 
(elliptic) \p-Laplace
equation $\Delta_p u=0$.

\begin{proof}
The proof of Theorem~6.5 in \cite{KiLi96} can immediately be modified to use the barrier family characterisation. 
Suppose that $x_0$ is regular for \p-harmonic functions
with respect to $G$ 
(for more details on this notion, see for example 
Heinonen--Kilpel\"ainen--Martio~\cite{HeKiMa}). 
Let $ \vp(x)= \abs{x-x_0}$ and let $u_j$ be a solution to
\[
\begin{cases}
\Delta_p u_j
=-j^{p-1}& \text{in } G,\\
u_j-j \vp\in W^{1,p}_0(G).
\end{cases}
\]
Then $u_j$ is \p-superharmonic and $u_j(x)\ge j \abs{x-x_0}$ because $\vp$ is \p-subharmonic. Define
\[
\begin{split}
w_j(x,t)=u_j(x)+j^{p-1} (t_0-t).
\end{split}
\]
Then
\[
\Delta_p w_j
=-j^{p-1}=\partial_t w_j,
\]
and it follows that $\{w_j\}_{j=1}^\infty$
is the desired barrier family, and $\xi_0$ is regular with respect to $G_T$.

The other direction of the proof of Theorem~6.5 in \cite{KiLi96} holds verbatim.
\end{proof}

We will also need the following result, stating that what happens in the future does not affect the regularity of the boundary point. To be more precise, if we split the domain $\Theta$ at the level $t_0$, and consider a boundary point $(x_0,t_0)$ at the same time instant, then the lower part
\[
\begin{split}
\Thetam=\{(x,t) \in \Theta : t < t_0\}
\end{split}
\]
determines the regularity. 
We begin with an introductory lemma, which we will also use later on.
In particular, it shows that the earliest points are always regular.

\begin{lemma}\label{lm-Omnibus}
Let $\xi_0=(x_0,t_0) \in \bdy \Theta$. If
$\xi_0\notin\bdy \Thetam$ 
\textup{(}in particular, if $\Thetam=\emptyset$\textup{)}, 
then $\xi_0$ is regular.
\end{lemma}

\begin{proof}
If $\xi_0\notin \partial \Thetam$, then the functions
\begin{equation*} 
  f_j(x,t)= j \frac{p-1}{p} |x-x_0|^{p/(p-1)} +nj^{p-1} (t-t_0)
\end{equation*}
are \p-parabolic in $\R^{n+1}$ and form a strong
barrier family in $\Theta\cap V$ for some neighbourhood $V$ of $\xi_0$.
\end{proof}

It follows from Theorem~\ref{thm-ell-reg-iff-par-reg} and 
Lemma~\ref{lm-Omnibus} that in the setting of Theorem~\ref{thm-cont-exist},
the Perron solution coincides with the one provided by Theorem~\ref{thm-cont-exist}.

\begin{theo} \label{thm-Omminus}
Let $\xi_0=(x_0,t_0) \in \bdy \Theta$ and  $\Thetam\ne\emptyset$. Then
$\xi_0$ is regular with respect to $\Theta$
if and only if either $\xi_0$ is regular with respect to $\Thetam$
or $\xi_0 \notin \bdy \Thetam$.
\end{theo}

\begin{proof}
Suppose that $\xi_0$ is a regular boundary point with respect to $\Theta$. 
Then by Proposition~\ref{prop-restrict}, either 
$\xi_0\notin\partial\Thetam$ 
or $\xi_0$ is also regular with respect to $\Thetam$.

To prove the converse, due to Lemma~\ref{lm-Omnibus}, we may assume that 
$\xi_0\in \partial \Thetam$ is a regular boundary point of $\Thetam$. Let
$w_j=\lP \psi_j$, where $\psi_j$ are as in 
\eqref{eq-barrier-const}.
As in \cite{KiLi96}, it can be shown that the restriction to $\Thetam$ of $w_j$ 
is the upper Perron solution of $\psi_j$ in $\Thetam$, and hence
\[
\lim_{\Thetam\ni \xi\to \xi_0} w_j(\xi)=\psi_j(\xi_0)=0.
\]
This also implies that $w_j$, extended by $\psi_j$ to $\bdry\Theta$, is continuous at
$\xi_0$.
Again following \cite{KiLi96}, it can be shown that the restriction to
$\Theta_\limplus=\{(x,t) \in \Theta : t > t_0\}$ of $w_j$ coincides with the lower Perron
solution of (the above extension) $w_j$ with respect to 
$\Theta_\limplus$.
Since earliest points are regular by Lemma~\ref{lm-Omnibus}, we can conclude from this that
also
\[
\lim_{\Theta\setm\Thetam\ni \xi\to \xi_0} w_j(\xi)=0,
\]
showing that $w_j$ form a barrier family in $\Theta$.
\end{proof}

The following characterisation is a bit similar
in flavour to our barrier characterisation, in that
it deduces regularity from properties
of a countable family.

\begin{prop}
Let $\xi_0 \in \bdy \Theta$
and let $d \in C(\bdy \Theta)$ be 
such that $d(\xi_0)=0$ and $d(\xi)>0$ for 
$\xi \in \bdy \Theta \setm \{\xi_0\}$.
Then $\xi_0$ is regular if and only if
\[
    \lim_{\Theta \ni \xi \to \xi_0} \uP (jd)(\xi)=0
    \quad \text{for } j=1,2,\ldots.
\]
\end{prop}

A typical example is $d(\xi)=|\xi-\xi_0|^\al$ with $\al>0$.
In the elliptic \p-harmonic case,
a similar characterisation was given
by Bj\"orn--Bj\"orn~\cite{BB}, 
Theorem~4.2 and Remarks~6.2 (which can
also be found as 
Theorem~11.2 and Remark~11.12 in \cite{BBbook}).
In the \p-harmonic case only the limit of $\uP d$ is required,
and the same is also true for the heat equation.
We do not know whether one limit is sufficient 
in the \p-parabolic case, 
and this issue seems closely related to the problem
of whether one barrier is sufficient.

A corresponding characterisation (with a family of Perron solutions)
was recently given
for the elliptic variable exponent $p(\cdot)$-harmonic
functions by Adamowicz--Bj\"orn--Bj\"orn~\cite{ABB}, Theorem~7.1.

\begin{proof}
The necessity is obvious.
For the sufficiency, let $f \in C(\bdy \Theta)$
and $\eps >0$.
Then we can find $j$ such that
$f < jd + f(\xi_0) + \eps$ on $\bdy \Theta$.
Thus,
\[
    \lim_{\Theta \ni \xi \to \xi_0} \uP f (\xi)
    \le f(\xi_0) + \eps + 
     \lim_{\Theta \ni \xi \to \xi_0} \uP (jd)(\xi)
     = f(\xi_0) +\eps.
\]
Letting $\eps \to 0$ shows that
$    \lim_{\Theta \ni \xi \to \xi_0} \uP f (\xi) \le f(\xi_0)$.
Applying this also to $-f$ yields 
that 
$    \lim_{\Theta \ni \xi \to \xi_0} \uP f (\xi) = f(\xi_0)$,
and thus $\xi_0$ is regular.
\end{proof}

%%%%%%%%%%%%%%%%%%%%%%%%%%%%%%%%%%%%%
\section{The exterior ball condition 
}\label{sect-ball-all-p}

In this section, we will show that if the domain satisfies 
an exterior ball condition, then it is regular. 
However, the ``south and north poles'' have to be excluded. 
To be more precise, the condition  $x_1\ne 0$ in 
Proposition~\ref{prop:exterior-ball-condition} below is meant to exclude 
both the south pole $(x_1,t_1-R_1)$ and the north pole $(x_1,t_1+R_1)$ of 
the exterior ball as a tangent point. The restriction on
the southern pole was already pointed out by Kilpel\"ainen and Lindqvist (see the comment at the end
of the proof of Theorem~6.2 in \cite{KiLi96}). Indeed, the top of
a cylindrical domain gives a natural counterexample.   
As for the north pole, when $x$ is close to $x_1=0$, $\Delta_p u$ is positive, and, curiously enough, the argument does not work.

Since the space is homogeneous and the \p-parabolic equation is translation
invariant, the geometric conditions implying regularity are the same at all points.
We therefore describe conditions for regularity of the origin from now on.

\begin{prop} \textup{(Exterior ball condition)}
\label{prop:exterior-ball-condition}
Let $\xi_0=(0,0)\in\partial\Theta$.
Suppose that there exists a  ball 
$B_1=B(\xi_1,R_1)$, 
with centre $\xi_1=(x_1,t_1)$ and radius $R_1$,
such that $B_1 \cap \Theta = \emptyset$ and
$\xi_0 \in \bdy B_1 \cap \bdy \Theta$.
If $x_1\ne 0$
then $\xi_0$ is regular with respect to $\Theta$.
\end{prop}

\begin{proof} 
Since the case $p=2$ is classical, we assume that $p \ne 2$.
We follow the ideas introduced in Kilpel\"ainen--Lindqvist~\cite{KiLi96}. 

Let $\xi_2=(x_2,t_2)=\frac{1}{2}\xi_1$ and $R_2=\frac{1}{2}R_1$.
Note that $R_1=|\xi_1|$ and $R_2=|\xi_2|$.
Let $\de=\frac{1}{2}|x_2|>0$ and $\Theta_0=\Theta \cap B(\xi_0,\de)$.
Here we have used that $x_1 \ne 0$.

Let $\xi=(x,t) \in \overline{\Theta}_0$
and $R=|\xi-\xi_2|\le  2R_2$.
Also let $j_0 \ge (n+p-2)/(p-1)\delta^2$  
and $j\ge j_0$ be integers.
Define
\[
w_j(\xi) = \gamma (e^{-j R_2^2}-e^{-j R^2}),
\]
where
$\gamma=\gamma(j)>0$ will be chosen later. 
As in \cite{KiLi96}, easy calculations yield
\[
\partial_t w_j = 2j \gamma e^{-j R^2} (t-t_2)
             \ge -4j \gamma R_2 e^{-j R^2}
\]
and 
\begin{align*}
\Delta_p w_j
&= (2j\gamma  )^{p-1} |x-x_2|^{p-2} e^{-j(p-1)R^2} [n+p-2-2j(p-1)|x-x_2|^2]\\
&\le (2j\gamma  )^{p-1} |x-x_2|^{p-2} e^{-j(p-1)R^2} [n+p-2-2j(p-1)\de^2].
\end{align*}
The choice of $j$ implies that $n+p-2-2j(p-1)\de^2 \le -j(p-1)\de^2$.
Since $\de\le|x-x_2|\le R\le 2R_2$, this gives
\begin{align*}
\Delta_p w_j
&\le - (2j\gamma  )^{p-1} |x-x_2|^{p-2} j(p-1) \de^2
          e^{-j(p-1)R^2}\\
&\le - C_0 (2j\gamma  )^{p-1} j (p-1) e^{-j(p-1)R^2},
\end{align*}
where 
\[
C_0=\begin{cases}
  (2R_2)^{p-2}\de^2, &  1<p<2, \\
  \de^p, & p>2.
  \end{cases}
\]
In order to have
$
\Delta_p w_j  \le \partial_t w_j
$
it is enough to verify that
\[
4j \gamma  R_2 e^{-j R^2}
   \le C_0 (2j\gamma )^{p-1} j (p-1) e^{-j(p-1)R^2},
\]
which is equivalent to
\begin{equation}
\gamma ^{p-2} \ge \frac{j^{1-p}R_2 e^{j(p-2)R^2}}{2^{p-3}(p-1) C_0}
=: C_1 j^{1-p} e^{j(p-2)R^2},
\label{eq-w-la-superp}
\end{equation}
where $C_1=R_2/2^{p-3}(p-1)C_0$.
Choose now
\begin{equation}
\gamma =\gamma (j) = \begin{cases} 
           (C_1 j^{1-p})^{1/(p-2)} e^{j R_2^2},  &1<p<2, \\
           (C_1 j^{1-p})^{1/(p-2)} e^{4j R_2^2}, &p>2.
                 \end{cases} 
\label{eq-choose-la}
\end{equation}
As $B(\xi_2,R_2) \cap \Theta_0$ is empty, we  have that $R_2\le R\le 2R_2$, 
which shows that \eqref{eq-w-la-superp} holds.
Therefore, $w_j$ is \p-superparabolic in $\Theta_0$.

We next want to show that $w_j$ satisfies \ref{cond-second-weak} of Definition~\ref{def-barrier},
and thus that $\{w_j\}_{j=j_0}^\infty$ is a barrier family,
condition \ref{cond-third} being immediate.

Let $\beta$ be the angle
between the vectors $-\xi_1$ and $\xi-\xi_1$,
$r_0=|\xi|$ and $r_1=|\xi-\xi_1| \ge R_1$.
The cosine theorem yields
$
  r_0^2 = r_1^2 + R_1^2 - 2r_1R_1\cos\beta.
$
Together with the cosine theorem again and the inequality $r_1 \ge R_1$
 this yields
\begin{align*}
   R^2-R_2^2 & = \bigl(\tfrac{1}{2} R_1\bigr)^2 + r_1^2 -r_1 R_1 \cos \beta 
                  -\bigl(\tfrac{1}{2} R_1\bigr)^2\\
       & = r_1^2 -\tfrac{1}{2}(r_1^2+ R_1^2 - r_0^2) 
        = \tfrac{1}{2} r_1^2 - \tfrac{1}{2} R_1^2 +\tfrac{1}{2} r_0^2
        \ge \tfrac{1}{2} r_0^2.
\end{align*}
It follows that for $\xi\in\overline{\Theta}_0 \setm B(\xi_0,r)$,
\begin{align*}
w_j(\xi) &= \gamma  e^{-j R_2^2} (1-e^{j (R_2^2-R^2)})
\ge \gamma  e^{-j R_2^2} (1-e^{-j r^2/2}).
\end{align*}
Inserting the expression \eqref{eq-choose-la} for $\gamma $ we obtain
\begin{align*}
w_j(\xi) &\ge  \begin{cases} 
           (C_1 j^{1-p})^{1/(p-2)} (1-e^{-j r^2/2}),  &1<p<2, \\
           (C_1 j^{1-p})^{1/(p-2)} e^{3j R_2^2}(1-e^{-j r^2/2}), &p>2.
                  \end{cases} 
\end{align*}
For a fixed $r$, the right-hand sides tend to $\infty$, as $j\to\infty$,
showing that $\{w_j\}_{j=j_0}^\infty$ is a 
barrier family with respect to $\Theta_0$.
Thus, by Theorem~\ref{thm:barrier-char},
$\xi_0=(0,0)$ is regular with respect to $\Theta_0$,
and hence, by Proposition~\ref{prop-local},
it is regular with respect to $\Theta$.
\end{proof}

%%%%%%%%%%%%%%%%%%%%%%%%%%%%%%%%%%%%%%%%%%%%%%%%%
\subsection{Regularity at the ``north pole''} %\label{S:origin}

The proof above for the exterior ball condition does not work
when $\xi_0=(0,0)$ is the north pole of the ball. Here we discuss some simple
sufficient conditions for the regularity of $\xi_0$, 
when it can be considered as some
sort of north pole for proper sets touching $\partial\Theta$ from below.

Observe that if we have a flat bottom, i.e.\ the half-space $\{(x,t): t<0\}
\subset \R^{n+1} \setm \Theta$, then Lemma~\ref{lm-Omnibus} 
gives the regularity.

\begin{prop}\label{prop:northpole-p>2}
Let $\Theta\subset\Rno$ be an open set and $(0,0) \in \bdy \Theta$.
Assume that for some $\theta>0$, 
\[
    \Theta \subset   \{(x,t): t>-\theta \abs{x}^l\},
\]
where $l\ge p/(p-1)$ if $1<p<2$, and $l>p$ if $p>2$.
Then $(0,0)$  is regular with respect to $\Theta$.
\end{prop}

\begin{proof}
By  Propositions~\ref{prop-restrict} and~\ref{prop-local} and
Theorem~\ref{thm-Omminus}
we may assume that
\[
   \Theta =   \{(x,t): t>-\theta\abs{x}^l \text{ and } -1 <t<0\}.
\]
We localise the problem by considering
\[
G^j=\biggl\{(x,t): |x|< \frac{1}{j^{1/k}}
     \text{ and } -\frac{\theta}{j^{l/k}}< t< 0\biggr\} 
\subset \Theta,
\]
where $k>0$ will be fixed later. Notice that
the set $G^j$ gets smaller, as $j$ grows.
Now let 
\begin{equation*} 
  f_j(x,t)= j \frac{p-1}{p} |x|^{p/(p-1)} +nj^{p-1} t,
\end{equation*}
which is \p-parabolic in $\R^{n+1}$. 
If $1<p<2$, then $f_j$ is positive in $G^j\cap\Theta$, provided $l\ge p/(p-1)$ and $j$ is large enough. If $p>2$, then $f_j$ is positive in $G^j\cap\Theta$, provided
$l>p/(p-1)+k(p-2)$ and $j$ is large enough.
Define
\[
   m_j:=\inf_{\Theta \cap \bdy G^j} f_j
      = j \frac{p-1}{p} \biggl(\frac{1}{j^{1/k}}\biggr)^{p/(p-1)} 
               -nj^{p-1} \frac{\theta}{j^{l/k}}.
\]
We want $m_j \to \infty$ as $j \to \infty$,
and this happens if 
\[
1-\frac{p}{(p-1)k} >0 \quad \text{and} \quad
  1-\frac{p}{(p-1)k} > p-1-\frac{l}{k}.
\]
The first condition is satisfied if $k>p/(p-1)$, and
the second condition is satisfied if  $l>p/(p-1)+k(p-2)$. 
This holds for $1<p<2$, $l\ge p/{(p-1)}$ and all $k>p/(p-1)$,
whereas for $p\ge2$ and $l>p$, this is true, provided we choose $k$ sufficiently 
close to $p/(p-1)$. 

Now let
\[
    h_j=\begin{cases}
             \min\{f_j,m_j\} & \text{in } G^j, \\
              m_j & \text{in } \Theta \setm G^j.
              \end{cases}
\]
Since $f_j$ is \p-parabolic in $\R^{n+1}$, the pasting lemma 
(Lemma~\ref{lem-pasting}), applied to $m_j$ and $f_j|_{G^j}$, shows that
$h_j$ is a positive (continuous) \p-superparabolic function
in $\Theta$, if $j$ is large enough.
As $m_j \to \infty$ and $G^j$ shrinks to $(0,0)$,
it follows that $\{h_j\}_{j=m}^\infty$ is a strong barrier family 
if $m$ is large enough,
and thus  
$(0,0)$ is regular with respect to $\Theta$.   
\end{proof}

\section{Exterior cone conditions}\label{S:cone}

Let us consider $\xi_0 \in \bdy \Theta$.
Without loss of generality we assume that $\xi_0=(0,0)$.
Assume that there is an open cone $C$
in the exterior $\R^{n+1} \setm \Theta$ with vertex at $\xi_0$.
When dealing with a space-time cylindrical domain, it is well-known that the 
nonparabolic part of the boundary is irregular; therefore, 
it is necessary that the cone $C$ contains some point $(x_C,t_C)$ with 
$t_C \le 0$
(or equivalently a point with $t_C < 0$) in order to have
a chance of implying the regularity of $\xi_0$.

As regularity is a local property by Proposition~\ref{prop-local},
we assume that the full cone is in the exterior, and we may therefore
assume that $|(x_C,t_C)|=1$.
Since $C$ is open, we can find an open cylindrical subcone $C_0$ with
vertex $\xi_0$ containing $(x_C,t_C)$ and such that
$\overline{\Theta} \cap \itoverline{C}_0 = \{\xi_0\}$.

To work out concrete examples to begin with, we show in 
Proposition~\ref{prop-cone-1+1} 
that the exterior cone condition in $\R^{1+1}$ 
yields regularity for $1<p<\infty$. 
We then use this to obtain a more general criterion
ensuring regularity in the $1+1$-dimensional case, see
Theorem~\ref{thm-1+1-gen}.
Then in Section~\ref{sect-gen-cone}, we focus on the higher-dimensional case.

\subsection{Exterior cone condition in 
  \texorpdfstring{$\R^{1+1}$}{} for  \texorpdfstring{$1<p<\infty$}{}}
\label{sec:1+1}

\begin{prop} \label{prop-cone-1+1}
\textup{(Exterior cone condition in $1+1$ dimensions)}
Let $\Theta \subset \R^{1+1}$ be bounded and $\xi_0=(0,0)\in\partial\Theta$.
Assume that there is an open cone $C \subset \R^{1+1} \setm \Theta$
with $\xi_0$ as vertex, and  a point $\xi=(x_C,t_C) \in C$ 
with $t_C \le 0$.
Then $\xi_0$ is regular.
\end{prop}

Case~1 below follows from the exterior ball condition or from
Proposition~\ref{prop-horiz}, and so only Case~2 is actually needed. 
However, because Case~1 is more elementary here and provides concrete
examples of simple barrier families, we have decided to include it here.

\begin{proof} 
We may assume that the cone $C$
is a full cylindrical cone, that $\Theta \subset B(\xi_0,1)$
and that $\itoverline{C} \cap \overline{\Theta}= \{\xi_0\}$.
By Theorem~\ref{thm-Omminus}, we may also assume that
$\Theta \subset \{(x,t) : t<0\}$.

\medskip
\emph{Case} 1. \emph{Horizontal cone, i.e.\ $t_C=0$, see Figure~\ref{Fig:1}.}
We may assume that $x_C=1$, and thus
\[
     C         = \{(x,t) : |t| < \ga  x\}.
\]
where $\ga$ is a positive parameter. 
\begin{figure}[t]
\psfrag{x}{$x$}
\psfrag{t}{$t$}
\psfrag{cone}{$|t|=\ga x$}
\psfrag{inner}{$|t|<\ga x$}
\psfrag{outer}{$|t|>\ga x$}
\psfrag{boundary}{$\partial\Theta$}
\begin{center}
\includegraphics[width=.4\textwidth]{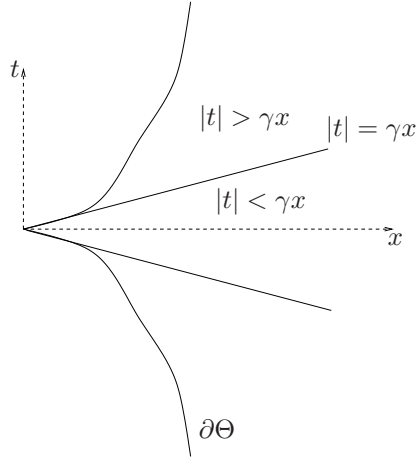}
\end{center}
\caption{The exterior cone condition with a horizontal cone.}\label{Fig:1}
\end{figure}
Let
\[
      u_{\mu,\alp}(x,t)=\mu (1-e^{-\alp(|t|-\ga x)}),
\]
where $\mu$ and $\alp$ are positive numbers.
We want $u_{\mu,\alp}$ to be \p-super\-pa\-ra\-bolic in $\Theta$.
We see that
\[
         \partial_x {u_{\mu,\alp}}= -\mu \alp\ga e^{-\alp(|t|-\ga x)},
\]
and hence
\begin{align*}
 \Delta_p  u_{\mu,\alp}
   & = - (\mu \alp\ga)^{p-1} \parts{}{x}  e^{-(p-1)\alp(|t|-\ga x)} \\
   & = - (\mu \alp\ga)^{p-1} (p-1)\alp\ga  e^{-(p-1)\alp(|t|-\ga x)}
   < 0\quad \text{in } \Theta.
\end{align*}
On the other hand,
\[
   \partial_t{u_{\mu,\alp}}= -\mu \alp e^{-\alp(|t|-\ga x)}
   \quad \text{in } \Theta,
\]
since $t<0$.
Thus, we have $\Delta_p  u_{\mu,\alp} \le \partial_t{u_{\mu,\alp}}$
if and only if
\[
    - (\mu \alp\ga)^{p-1} (p-1)\alp\ga  e^{-(p-1)\alp(|t|-\ga x)}
    \le - \mu \alp e^{-\alp(|t|-\ga x)}
\]
which is equivalent to
\begin{equation} \label{eq-mu-cond}
   \mu^{p-2}  \alp^{p-1} \ga^p (p-1) \ge e^{(p-2)\alp(|t|-\ga x)}.
\end{equation}
As before, we consider $p>2$ and $1<p<2$ separately, starting
with the former case, as it turns out to be simpler.

\medskip
\emph{Case} $1a$. $p>2$.
In this case we let $\alp=1$ and thus $u_{\mu,1}$ is
\p-superparabolic in $\Theta$ if
\[
   \mu  \ge \frac{e^{1+\ga}}{(\ga^p (p-1))^{1/(p-2)}} =:\mu_0,
\]
since $\Theta \subset B(\xi_0,1)$.
It follows that for $j\ge\mu_0$, $w_j=u_{j,1}$ is a suitable barrier family
and thus $\xi_0$ is regular.

\medskip
\emph{Case} $1b$. $1<p<2$.
In this case we let
$\al(\mu) := (\mu^{(2-p)}/\ga^p (p-1))^{1/(p-1)}$. 
Then \eqref{eq-mu-cond} is satisfied with $\al:=\al(\mu)$. Moreover, indexing with $j$ we have that
$w_j:=u_{j,\alp(j)}$ is \p-superparabolic in $\Theta$.
It also has the necessary limits at the boundary.
Finally, for sufficiently large $j$, we have
$\alp(j) \ge 1$  and hence for $(x,t) \in \overline{\Theta}$, 
\begin{align*}
     w_j(x,t)& = j(1-e^{-\al(j) (|t|-\ga x)}) 
        \ge j(1-e^{-(|t|-\ga x)}), 
\end{align*} 
showing that $w_j$ is a suitable barrier family,
and thus $\xi_0$ is regular.

\medskip
\emph{Case} 2. \emph{Downwards cone, i.e.\ $t_C<0$, see Figure\/~\ref{Fig:2}.}
(By downwards we do not mean straight downwards, i.e.\ $x_C$ is not
necessarily $0$.)

In this $2$-dimensional situation the cone splits $\Theta$ into two parts,
one to the left and one to the right of the cone, or more formally
into 
\begin{align*}
     \Theta_1 &=\{(x,t) \in \Theta:  x < y \text{ whenever }
              (y,t) \in C\}, \\
     \Theta_2 &=\{(x,t) \in \Theta :  x > y \text{ whenever }
              (y,t) \in C\}.
\end{align*}
It is possible that one of these is empty, but not both.

\begin{figure}[t]
\psfrag{cone}{$C$}
\psfrag{omegaone}{$\Theta_1$}
\psfrag{omegatwo}{$\Theta_2$}
\psfrag{vertex}{$\xi_0$}
\begin{center}
\includegraphics[width=.4\textwidth]{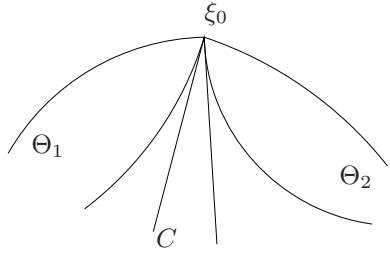}
\end{center}
\caption{The exterior cone condition with a downwards cone.}\label{Fig:2}
\end{figure}

It is easy to see that $\Theta_1$, if nonempty,
satisfies the exterior cone condition
with a horizontal cone around $\xi_0$,
and hence $\xi_0$ is regular with respect to $\Theta_1$ by Case~1,
and similarly with respect to $\Theta_2$.
It thus follows from Lemma~\ref{lem-component} that
$\xi_0$ is regular with respect to $\Theta$.
(If $\Theta_1 =\emptyset$ or $\Theta_2=\emptyset$ this follows directly.)
\end{proof}

We can considerably generalise the previous result.

\begin{theo} \label{thm-1+1-gen}
Let $C =\{(x,t): |x| \le -\theta t \}\subset\R^{1+1}$
for some $\theta>0$,
and assume that $\xi_0=(0,0)\in\partial\Theta$ 
and that there is a closed nonsingleton 
connected set $A \subset C \setm \Theta$ containing $\xi_0$.
Then $\xi_0$ is regular.
\end{theo}

\begin{proof}
Let $(x_A,t_A) \in A \setm \{\xi_0\}$ so that $t_A<0$.
Since regularity is a local property, we may assume
that $\Theta \subset B(\xi_0,-t_A)$.
By Theorem~\ref{thm-Omminus}, we may assume that
$\Theta \subset \{(x,t) : t<0\}$.
We next split $\Theta$ into two parts,
\begin{align*}
     \Theta_1 &=\{(x,t) \in \Theta :  x < y
               \text{ whenever }      (y,t) \in A\}, \\
     \Theta_2 &=\Theta \setm \Theta_1.
\end{align*}
It is possible that one of these is empty, but not both.
We can now conclude the proof exactly as in the proof of Case~2 
of Proposition~\ref{prop-cone-1+1}.
\end{proof}

\subsection{Generalised horizontal cone condition in 
\texorpdfstring{$\R^{n+1}$}{}}
\label{sect-gen-cone}

In this section we will show that if the lower
half of a horizontal cone is in the complement
then the point is regular.

\begin{prop} \label{prop-horiz}
\textup{(Horizontal cone condition)}
Let  
$\Theta\subset\Rno$ be an open set and $\xi_0=(0,0)\in\partial\Theta$.
If there exist $\theta,r >0$ and a unit vector $v \in \R^n$ such
that the cone
\begin{equation}   \label{eq-def-cone-by-v}
    C:= \{(x,t): |(x,t)| < \theta x \cdot v, \ t \le 0  
\text{ and } |x|<r\}
    \subset \Rno \setm \Theta,
\end{equation}
then $\xi_0$ is regular with respect to $\Theta$. 
\end{prop}

Relying on potential theoretic tools we can deduce a fairly
more general result, see Proposition~\ref{prop-sect-gen-cone} below.
Proposition~\ref{prop-horiz} then follows directly as a special case,
by taking $\eta(t)=\ga t$ and
$E=\{x\in\R^n: |x|<\theta' x\cdot v\}$ 
for some sufficiently small $\ga,\theta'>0$,
and upon observing that the corresponding set $E'$ defined by \eqref{def-E-prime}
below is then contained
in the cone $C$ in~\eqref{eq-def-cone-by-v}.

Proposition~\ref{prop-horiz} holds for any $p>1$, and in particular for $p=2$. 
It is immediate to see that a horizontal cone $C$ such as defined by \eqref{eq-def-cone-by-v},
always contains a so-called \emph{tusk},
that is a set in $\Rno$ of the form
\[
V:= \{(x,t): -T<t<0 \text{ and } |x-(-t)^{1/2}x_0|^2 < R^2(-t)\},
\]
for some $x_0\in\Rn \setm \{0\}$, and positive constants $R$ and $T$, 
provided $T$ is small enough.
Consider $\Theta\in\Rno$ and $\xi_0=(0,0)\in\partial\Theta$: it is well known that if there is a tusk $V$
with $\overline{V} \cap \overline{\Theta}=\{\xi_0\}$,
then $\xi_0$ is regular for the heat equation 
(see Effros--Kazdan~\cite{EfKaz}, which refers to $\xi_0$ as being 
\emph{parabolically touchable}, and 
Lieberman~\cite{lieberman}). Therefore, under this point of view, Proposition~\ref{prop-horiz} gives a weaker condition. We do not know, whether a proper tusk condition holds for the parabolic \p-Laplacian, when $p\not=2$.

Recall that for $1<p\le n$
a set $A\subset\Rn$ is \emph{\p-thick} at $0\in\Rn$ if
\[
\int_0^1\biggl(\frac{\operatorname{cap}_p(A\cap B'(0,r),B'(0,2r))}{r^{n-p}}\biggr)^{1/(p-1)}
    \frac{dr}r<\infty,
\]
and \emph{\p-thin} otherwise,
where for a given set $E\subseteq B'(0,2r)$, $\operatorname{cap}_p(E,B'(0,2r))$ is the variational \p-capacity of $E$, and
$B'(0,r)$ is the ball of centre $0$ and radius $r$ in $\Rn$.

The \emph{Wiener criterion} says that $0$ is regular
for \p-harmonic functions with respect to an open set $V$ if and only
if $\R^n \setm V$ is \p-thick at $0$.
It  was obtained by Wiener~\cite{wiener} for $p=2$ in $\R^3$.
In the nonlinear case, the sufficiency was obtained
by Maz{\cprime}ya~\cite{mazya70} and Gariepy--Ziemer~\cite{GaZi}, 
and the necessity by Lindqvist and Martio~\cite{lindqvistm85} 
(for $p\ge n-1$) and by Kilpel\"ainen and Mal\'y~\cite{KilpMaly},
see also  Chapter~4 of Mal\'y--Ziemer~\cite{MaZi}.
These results have later been proved in more general settings in e.g.\
 \cite{mikkonen}, \cite{BMS}, \cite{JB-Matsue} and \cite{JB-pfine}.

Note that for $p>n$, by the same criterion, a singleton is always
regular, which is useful and reflected in Example~\ref{ex-regular-singleton} below.

\begin{prop}\label{prop-sect-gen-cone}
Let $\Theta\subset\Rno$ be an open set and $\xi_0=(0,0)\in\partial\Theta$.
Also let  $E\subset\R^n$ be a closed set which is 
\p-thick at $0$. 
Assume that for some $\theta>0$,
some $\zeta\in\R^n$ with
$|\zeta|=1$, and a nondecreasing differentiable
function $\eta:[0,r_0/\theta]\to[0,\infty)$,
such that $\eta(0)=0$ and $\eta'\le\theta$ in $[0,r_0/\theta]$, we have
\begin{equation}\label{def-E-prime}
E':= \biggl\{(x,t)\in B(\xi_0,r_0): -\frac{r_0}\theta<t<0 
  \text{ and } x -\eta(-t)\zeta \in E\biggr\} 
  \subset  B(\xi_0,r_0) \setm\Theta.
\end{equation}
Then $\xi_0$ is regular with respect to $\Theta$.
\end{prop}

Before the proof, let us point out some special cases.
Proposition~\ref{prop-sect-gen-cone} gives us a sufficient condition for the boundary regularity of
$\xi_0\in\bdry\Theta$, and covers a wide range of situations,
the most important one probably being the one already
described above in Proposition~\ref{prop-horiz}. 
Here are some other
interesting cases.

\begin{example}  \label{ex-regular-singleton}
In the simplest case $n=1$, every $x_0\in\R$ is regular,
so one can take $E=\{x_0\}$.
Regularity of $\xi_0$ is therefore guaranteed, whenever there is
a curve in $\R^{1+1}$ approaching $\xi_0$ from below, i.e.\
with subhorizontal derivative at $\xi_0$;
cf.\ Theorem~\ref{thm-1+1-gen}.
\end{example}

\begin{example}
If $n=2$, then a segment in $\R^2$ with $x_0$ as an endpoint  will do
as $E$.  Then $E'$ will be a vertical two-dimensional triangle
with vertex at $\xi_0$.
More generally, if $\ga$ is a continuous curve in $\R^2$,
ending at $x_0$, then $B'(x_0,r)\setm\ga$ is regular 
for \p-harmonic functions at $x_0$
and the corresponding set $E'$ is a two-dimensional triangle-shaped
``curtain'' following the curve $\ga$.
It is also possible to have disconnected $E$, i.e.\ $E$ consisting of
suitably chosen segments accumulating at $x_0$.
\end{example}

\begin{example}
In higher dimensions, the easiest generalisation is a set $E$ which satisfies
an interior cork-screw condition at $x_0$ (the cone considered above does).
More generally, it is enough if $B'(x_0,r)\setm E$ is porous at $x_0$.
We recall that a set $M\subset\Rn$ is \emph{porous} at a point 
$x\in\Rn$, if there exists a constant $c\in(0,1)$ such that, 
for each $\eps>0$, there exist $y\in\Rn$ and $r>c|x-y|$, with $|x-y|<\eps$ and 
$M\cap B'(y,r)=\emptyset$.
See Corollary~11.25 in Bj\"orn--Bj\"orn~\cite{BBbook} for some
more general sufficient porosity conditions.
A sharp condition for $E$ follows from the Wiener criterion.
\end{example} 

\begin{proof}[Proof of Proposition \ref{prop-sect-gen-cone}]
Let $E\subset\R^n$ be closed and
such that $0\in\bdry E$.
Assume that $E$ is \p-thick at $0$.
Let 
$\Om=B'(0,r_0)\setm E \subset \R^n$.
Then $0$ is regular for \p-harmonic functions with
respect to $\Om$,
by the Wiener criterion (see above). 

Let $u_j$
be a continuous solution to the Dirichlet problem
\[ 
\begin{cases}
\Delta_p u_j(x) = \theta \min\{\grad u_j(x)\cdot \zeta, 0\} -j
&  \text{in } \Om, \\
u_j-f_j \in \Wp_0(\Om),
\end{cases}
\] 
where $f_j(x)=j|x|$, $j=1,2,\ldots$.
Such a function $u_j$ exists by Theorem~6.21 in 
Mal\'y--Ziemer~\cite{MaZi},
and Corollary~6.22 in~\cite{MaZi} implies that
\begin{equation}
\lim_{x\to 0}u_j(x)= f_j(0) =0,
\label{eq-u-to-0}
\end{equation}
since $0$ is regular for \p-harmonic functions.
Moreover, $\Delta_p u_j<0$, i.e.\ 
$u_j$ is a supersolution to the \p-harmonic equation.
Since $f_j$ is a subsolution to the \p-harmonic equation
(by direct calculation),
Lemma~3.18 in 
Heinonen--Kilpel\"ainen--Martio~\cite{HeKiMa}
(and the continuity of $u_j$ and $f_j$) shows that
$u_j\ge f_j$ in $\Om$.
Let
\[
\Theta'= \{(x,t): - r_0/\theta< t< 0 \text{ and }
  x - \eta(-t)\zeta \in \Om \}
\]
 and for $\xi=(x,t)\in\Theta'$ define
\begin{equation}   \label{eq-def-vj}
v_j(\xi) = u_j(x-\eta(-t)\zeta) -jt.
\end{equation}
Then
\[
\bdry_t v_j (\xi) = \eta'(-t)
            \grad u_j(x-\eta(-t)\zeta) \cdot\zeta - j
\]
and
\[
\Delta_p v_j(\xi) = \Delta_p u_j(x-\eta(-t)\zeta)
      = \theta \min\{\grad u_j(x-\eta(-t)\zeta)\cdot \zeta, 0\} - j < 0.
\]
Since $0 \le\eta'\le\theta$, we thus have
\begin{align*}
\Delta_p v_j(\xi)
     &\le \eta'(-t)
          \min\{\grad u_j(x-\eta(-t)\zeta)\cdot \zeta, 0\} -j\\
&\le \eta'(-t) \grad u_j(x-\eta(-t)\zeta)\cdot \zeta -j
= \bdry_t v_j (\xi),
\end{align*}
i.e.\ $v_j$ is \p-superparabolic in $\Theta'$.
As $u_j\ge f_j$ in $\Om$, we have
\[
v_j(\xi) \ge j ( |x -\eta(-t)\zeta| -t ) =: j {d}(\xi).
\]
Moreover, \eqref{eq-u-to-0} and~\eqref{eq-def-vj} yield
\[
\lim_{\xi\to \xi_0} v_j(\xi) = 0.
\]
Hence $v_j$ forms  a strong family of barriers at $\xi_0=(0,0)$
with respect to $\Theta'$, and Theorem~\ref{thm:barrier-char} shows that
$\xi_0$ is a regular boundary point with respect to $\Theta'$.

By our assumptions, for some
$0<r<\tfrac12\min\{r_0,r_0/\theta\}$ we have that
$\Theta_\limminus\cap B(\xi_0,r) \subset \Theta'\cap B(\xi_0,r)$.
Then $\xi_0$ is a regular boundary point with respect to $\Theta'\cap B(\xi_0,r)$,
and hence also with respect to $\Theta_\limminus$ and consequently with respect to $\Theta$,
by Propositions~\ref{prop-restrict},~\ref{prop-local}
and Theorem~\ref{thm-Omminus}.
\end{proof}

%%%%%%%%%%%%%%%%%%%%%%%%%%%%%%%%%%%%%%%%%%%%%%%
\section{A \texorpdfstring{Petrovski\u\i}{Petrovskii} condition for 
\texorpdfstring{$p>2$}{}}\label{S:petr}

In this section we show that there is an increasing barrier family in the domain \eqref{Eq:petr:1} at $(0,0)$, and thus the origin as final point is regular. 
We follow the same approach as in 
Lindqvist~\cite{lindqvist95}: the only difference is in the choice of the function $f$ in \eqref{eq:petrowsky-barriers}. 

\begin{prop}\label{Prop:Petr-deg}
Let $p>2$. The origin $(0,0)$ is regular with respect to the domain
\begin{equation}\label{Eq:petr:1}
\begin{aligned}
\Theta=\biggl\{(x,t):&
  -\frac{1}{2e}<t<0 \text{ and }\\
   &\biggl(\frac{|x|}{(-t)^{1/\lambda}}\biggr)^{p/(p-1)}
   <K(-t)^{n(p-2)/\lambda}h(t)^{\al(p-2)}\biggr\},
\end{aligned}
\end{equation}
in $\Rno$, 
where $K$ and $\al$ denote arbitrary positive constants, $\lm=n(p-2)+p$
and 
\[
    h(t)=\frac{|{\log(-t)}|^{p-2}-1}{p-2}.
\]
\end{prop}

The choice $T=-\frac1{2e}$ is completely immaterial: any negative value would do, 
as regularity is a purely local property by Proposition~\ref{prop-local}. 
In the interval $\bigl(-\frac1e,0\bigr)$, 
the function $h$ is strictly positive, 
and this simplifies some of the calculations to come.

\begin{remark}
If we choose $\al=\frac1{p-2}$, then in \eqref{Eq:petr:1} we obtain
\[
\biggl(\frac{|x|}{(-t)^{1/\lm}}\biggr)^{p/(p-1)}<K(-t)^{n(p-2)/\lm}\frac{|{\log(-t)}|^{p-2}-1}{p-2},
\]
and when $p\to2$, this \emph{formally} becomes 
\[
\biggl(\frac{|x|}{(-t)^{1/2}}\biggr)^{2}<K\log|{\log(-t)}|,
\]
which resembles Petrovski\u\i's condition mentioned in the introduction.
Unfortunately, this result is purely formal, because when $p\to2$, 
we have $M\to\infty$ and $\eps\to0$ in the proof below, 
and the argument becomes void. The lack of stability in our estimates is apparent also from another point of view: the constant $K$ is completely arbitrary here, whereas from Petrovski\u\i's condition
it is known that its value is very precisely determined.
\end{remark}

\begin{remark}
In Kilpel\"ainen--Lindqvist~\cite{KiLi96}, pp.\ 676--677,
 it is shown that the origin is an irregular boundary point with respect to the so-called Barenblatt balls; namely, it is shown that $(0,0)$ is an irregular boundary point with respect to the domain
\[
 \biggl\{(x,t):
\frac{p-2}{p\lm^{1/(p-1)}}\biggl(\frac{|x|}{(-t)^{1/\lm}}\biggr)^{p/(p-1)}
   <1-2^{-(p-2)/(p-1)} 
   \text{ and } -T <t< 0
\biggr\},
\]
where $T$ depends on $p$.
Now, provided $t$ is small enough, which we can always assume without loss of generality, as regularity is a local property by Proposition~\ref{prop-local}, it is easy to check that in $\Theta$, 
\begin{align*}
\frac{p-2}{p\lm^{1/(p-1)}}\biggl(\frac{|x|}{(-t)^{1/\lm}}\biggr)^{p/(p-1)}
&< K \frac{p-2}{p\lm^{1/(p-1)}}(-t)^{n(p-2)/\lm}h(t)^{\al(p-2)} \\
&   <1-2^{-(p-2)/(p-1)}. 
\end{align*}
This suggests that there is a sort of threshold for the regularity of the 
final point. 
Once more, there is no stability in this estimate, as $p\to2$.
\end{remark}

\begin{proof}[Proof of Proposition~\ref{Prop:Petr-deg}]
By Definition~\ref{def-barrier} and Theorem~\ref{thm:barrier-char}, 
it is enough to show that there exists a  barrier family 
$\{w_j\}_{j=1}^\infty$ in $\Theta$  
at the origin $\xi_0=(0,0)$.
The  family $\{w_j\}_{j=1}^\infty$  we construct 
will be smooth in $\Theta$, 
so that
${\partial_t w_j}-\Delta_p w_j\ge0$ is satisfied in the classical sense.
It will be constructed in the form
\begin{align}
w_j(x,t)& =f(t)\biggl[j+\frac{p-2}{p\lm^{1/(p-1)}}
     \biggl(\frac{|x|}{(-t)^{1/\lm}}\biggr)^{p/(p-1)}
\biggr]^{(p-1)/(p-2)} \nonumber \\
  &\quad -j^{(p-1)/(p-2)}f(t)+\rho_j(t),
\label{eq:petrowsky-barriers}
\end{align}
where
\[
   f(t)=-\eps h(t)^\alp 
     = -\eps\biggl(\frac{|{\log(-t)}|^{p-2}-1}{p-2}\biggr)^\alp<0,
\]
$\eps>0$ is a positive parameter to be chosen,
and $\rho_j>0$ is a proper function to be determined.
(The function $\rho_j$ will depend on $j$, but $\eps$ will not.)

We shall select $\rho_j$ such that $w_j$ is a supersolution 
in the domain where $w_j>0$,
and this domain is to contain $\Theta$. In the following, 
for simplicity we will drop the subscript  $j$ in $w_j$ and $\rho_j$.
Notice that $w$ is positive when
\begin{equation}\label{Eq:petr:2}
\biggl[j+\frac{p-2}{p\lm^{1/(p-1)}}
      \biggl(\frac{|x|}{(-t)^{1/\lm}}\biggr)^{p/(p-1)}
\biggr]^{(p-1)/(p-2)}<j^{(p-1)/(p-2)}-\frac{\rho(t)}{f(t)}.
\end{equation}
Moreover,
\[
w(x,t)<f(t)j^{(p-1)/(p-2)}-f(t)j^{(p-1)/(p-2)}+\rho(t)=\rho(t).
\]
Thus, \ref{cond-third} in Definition~\ref{def-barrier} holds true if $\rho(t)\to0$ as $t\to0^\limminus$. 
This requirement restricts the choice of $\rho$
in a decisive way. However, it turns out that $\rho$ can be chosen in such a way, 
that conditions \ref{cond-first}--\ref{cond-second-weak} in Definition~\ref{def-barrier}
 can all be satisfied. 

We now show that $w$ is \p-superparabolic in the domain defined by 
\eqref{Eq:petr:2}.
We will then prove that this domain contains $\Theta$, and at the same time 
that \ref{cond-second-weak} in Definition~\ref{def-barrier} is satisfied on $\partial\Theta$.

We set 
\begin{equation}\label{Eq:petr:3}
F(x,t)=j+\frac{p-2}{p\lm^{1/(p-1)}}
    \biggl(\frac{|x|}{(-t)^{1/\lm}}\biggr)^{p/(p-1)}.
\end{equation}
Hence
\begin{equation*}
\begin{aligned}
\nabla F &=\frac{p-2}{(p-1)\lm^{1/(p-1)}}
    \frac{|x|^{(2-p)/(p-1)}x}{(-t)^{p/\lm(p-1)}},\\
 \partial _t F &=\frac{p-2}{(p-1)\lm^{p/(p-1)}}
      \biggl(\frac{|x|}{(-t)^{1/\lm}}\biggr)^{p/(p-1)}\frac1{-t}
   =\frac{p (F(x,t)-j)}{\lm(p-1)(-t)}.
\end{aligned}
\end{equation*}
Since $w(x,t)=f(t)F(x,t)^{(p-1)/(p-2)}+\phi(t)$ with 
$\phi(t)=-j^{(p-1)/(p-2)}f(t)+\rho(t)$, we have
\begin{equation*}
\begin{aligned}
\nabla w &=f(t)\frac{F(x,t)^{1/(p-2)}}{\lm^{1/(p-1)}}
\frac{|x|^{(2-p)/(p-1)}x}{(-t)^{p/\lm(p-1)}},\\
|\nabla w|^{p-2} \nabla w
    &=|f(t)|^{p-2}f(t)\frac{F(x,t)^{(p-1)/(p-2)}}{\lm}\frac{x}{(-t)^{p/\lm}}.
\end{aligned}
\end{equation*}
Therefore,
\begin{align*}
\Delta_p w &=|f(t)|^{p-2}f(t)\frac{F(x,t)^{(p-1)/(p-2)}}{\lm}\frac{n}{(-t)^{p/\lm}}\\
                    &+|f(t)|^{p-2}f(t) \frac{F(x,t)^{1/(p-2)}}{\lm^{p/(p-1)}}
    \biggl(\frac{|x|}{(-t)^{1/\lm}}\biggr)^{p/(p-1)}\frac1{(-t)^{p/\lm}}\\
                     &= |f(t)|^{p-2}f(t)\frac{F(x,t)^{(p-1)/(p-2)}}{\lm}
         \frac{n}{(-t)^{p/\lm}}\\
   & \quad +\frac{p|f(t)|^{p-2}f(t)}{\lm (p-2)(-t)^{p/\lm}}F(x,t)^{1/(p-2)}(F(x,t)-j).
\end{align*}
Moreover,
\begin{align*}
\partial_t w &= \phi'(t)+f'(t)F(x,t)^{(p-1)/(p-2)}+\frac{p-1}{p-2}f(t)F(x,t)^{1/(p-2)}\partial_t F\\
      &= \phi'(t)+f'(t)F(x,t)^{(p-1)/(p-2)}
    +\frac{p f(t)F(x,t)^{1/(p-2)}(F(x,t)-j)}{\lm(p-2)(-t)}.
\end{align*}
Combining the previous expressions yields
\begin{align*}
\partial_t w-\Delta_p w 
&=\phi'(t)+f'(t)F(x,t)^{(p-1)/(p-2)}+\frac{p f(t)F(x,t)^{1/(p-2)}(F(x,t)-j)}{\lm(p-2)(-t)}\\
         &\quad -|f(t)|^{p-2}f(t)\frac{F(x,t)^{(p-1)/(p-2)}}{\lm}
                         \frac n{(-t)^{p/\lm}}\\
 &\quad -\frac{p|f(t)|^{p-2}f(t)}{\lm(p-2)(-t)^{p/\lm}}F(x,t)^{1/(p-2)}(F(x,t)-j)\\
        &= \phi'(t)-\frac{pj}{\lm(p-2)}
    \biggl(\frac1{-t}-\frac{|f(t)|^{p-2}}{(-t)^{p/\lm}}\biggr)f(t)F(x,t)^{1/(p-2)}\\
                           &\quad +F(x,t)^{(p-1)/(p-2)}
     \biggl(f'(t)+\frac{p f(t)}{\lm(p-2)(-t)}
    -\frac{n|f(t)|^{p-2}f(t)}{\lm(-t)^{p/\lm}}\\
     &\quad -\frac{p|f(t)|^{p-2}f(t)}{\lm(p-2)(-t)^{p/\lm}}\biggr).
\end{align*}
Since 
\[
 \frac n\lm+\frac p{\lm(p-2)}=\frac1{p-2},
\]
 we obtain
\begin{align}  \label{eq-wt-Delta-w}
\partial_t w-\Delta_p w &= \phi'(t)-\frac{pj}{\lm(p-2)}
   \biggl(\frac1{-t}-\frac{|f(t)|^{p-2}}{(-t)^{p/\lm}}\biggr)f(t)F(x,t)^{1/(p-2)}\\
      &\quad +F(x,t)^{(p-1)/(p-2)}\biggl(f'(t)+\frac{p f(t)}{\lm(p-2)(-t)}
         -\frac{|f(t)|^{p-2}f(t)}{(p-2)(-t)^{p/\lm}}\biggr). \nonumber
\end{align}
We need to ensure that $\partial_t w-\Delta_p w \ge0$.
As we mentioned before, we now choose
\begin{equation}   \label{eq-choose-f}
   f(t)=-\eps h(t)^\alp 
     = -\eps\biggl(\frac{|{\log(-t)}|^{p-2}-1}{p-2}\biggr)^\alp,
\end{equation}
where $\eps$ is still to be fixed. Since 
\[
f'(t)=-\al\eps h(t)^{\al-1}\frac{|{\log(-t)}|^{p-3}}{-t} <0,
\]
 we have
\begin{align*}
&f'(t)+\frac{p f(t)}{\lm(p-2)(-t)}-\frac{|f(t)|^{p-2}f(t)}{(p-2)(-t)^{p/\lm}}\\
&\qquad =-\al\eps h(t)^{\al-1}\frac{|{\log(-t)}|^{p-3}}{-t}
    -\frac{p\eps}{\lm(p-2)(-t)}h(t)^\al
+\frac{\eps^{p-1}}{(p-2)(-t)^{p/\lm}}h(t)^{\al(p-1)}.
\end{align*}
This expression is certainly negative if
\[
-\frac{p\eps}{\lm(p-2)(-t)}h(t)^\al
    +\frac{\eps^{p-1}}{(p-2)(-t)^{p/\lm}}h(t)^{\al(p-1)}\le0,
\]
which holds if 
\[
\frac p\lm\ge g(t)^{p-2}\eps^{p-2},
\quad \text{where }
g(t)=(-t)^{n/\lm}h(t)^\al.
\]
Letting $M=\sup_{-1/2e <t <0} g(t)$, which is finite and positive,
the negativity condition becomes
\begin{equation}\label{Eq:petr:4}
\frac p\lm=\eps^{p-2} M^{p-2}.
\end{equation}
We fix $\eps$ in this way; note that $\eps$ depends on $n$, $p$ and $\al$. 
Once $\eps$ is chosen, it is easy to check that
\[
-\frac{pj}{\lm(p-2)}\biggl(\frac1{-t}-\frac{|f(t)|^{p-2}}{(-t)^{p/\lm}}\biggr)f(t)>0,
\]
as $p/\lm<1$. 
Relying on \eqref{Eq:petr:3} and \eqref{eq-wt-Delta-w},
we have in the domain defined by \eqref{Eq:petr:2}, 
\begin{align*}
\partial_t w-\Delta_p w 
&\ge \phi'(t)-\frac{pj}{\lm(p-2)}\biggl(\frac1{-t}-\frac{|f(t)|^{p-2}}
    {(-t)^{p/\lm}}\biggr)f(t) j^{1/(p-2)}\\
           &\quad +\biggl(f'(t)+\frac{p f(t)}{\lm(p-2)(-t)}
        -\frac{|f(t)|^{p-2}f(t)}{(p-2)(-t)^{p/\lm}}\biggr)
         \biggl(j^{(p-1)/(p-2)}-\frac{\rho(t)}{f(t)}\biggr)\\
             &= \rho'(t)
    +\frac{n \eps^{p-1}j^{(p-1)/(p-2)}}{\lm(-t)^{p/\lm}}h(t)^{\al(p-1)}\\
      &\quad +\rho(t)\biggl(\frac{-\al|{\log(-t)}|^{p-3}}{-t h(t)}
   -\frac{p}{\lm(p-2)(-t)}
                  +\frac{\eps^{p-2}h(t)^{\al(p-2)} }{(p-2)(-t)^{p/\lm}}\biggr),
\end{align*}
where we have taken into account that $\phi(t)=\rho(t)-j^{(p-1)/(p-2)}f(t)$. 
Now we choose $\rho$ in such a way that the expression on the right-hand side is positive. We let
\begin{equation}  \label{eq-def-rho}
\rho(t)=A(-t)^{1-p/\lm}h(t)^{\al(p-1)}=A(-t)^{1-p/\lm}\efoft^{\al(p-1)},
\end{equation}
where $A>0$ is to be determined. 
As $p<\lambda$, $\lim_{t\to 0^\limminus} \rho(t)=0$ as required.
Since
\begin{align*}
\rho'(t) &= -A\frac{n(p-2)}{\lm}(-t)^{- p/\lm}h(t)^{\al(p-1)}\\
    &\quad +A\al(p-1)(-t)^{1-p/\lm}h(t)^{\al(p-1)-1}\frac{|{\log(-t)}|^{p-3}}{-t},
\end{align*}
we obtain
\begin{align*}
\partial_t w-\Delta_p w
&\ge \frac{1}{\lm(-t)^{p/\lm}}\biggl(-An(p-2)+n\eps^{p-1}j^{(p-1)/(p-2)}
       -\frac{Ap}{(p-2)}\biggr) h(t)^{\al(p-1)}\\
       &\quad+\frac{A\al(p-2)}{(-t)^{p/\lm}}h(t)^{\al(p-1)-1}|{\log(-t)}|^{p-3}
       +\frac{A\eps^{p-2}(-t)^{1-2p/\lm}}{p-2}h(t)^{\al(2p-3)}.
\end{align*}
We conclude that $\partial_t w-\Delta_p w\ge0$, provided we choose $A>0$ such that
\[
-A\biggl(n(p-2)+\frac p{p-2}\biggr)+n\eps^{p-1}j^{(p-1)/(p-2)}\ge0,
\]
which holds when we let
\begin{equation}\label{Eq:petr:5}
A=\frac{n(p-2)\eps^{p-1}j^{(p-1)/(p-2)}}{n(p-2)^2+p}.
\end{equation}
Relying on the choice of $\eps$, $A$, $f$ and $\rho$ in
\eqref{eq-choose-f}--\eqref{Eq:petr:5}, we can now rewrite \eqref{Eq:petr:2} in the following way
\begin{equation}\label{Eq:petr:6}
\begin{aligned}
&\biggl[1+\frac{p-2}{pj\lm^{1/(p-1)}}\biggl(\frac{|x|}{(-t)^{1/\lm}}
   \biggr)^{p/(p-1)}\biggr]^{(p-1)/(p-2)}\\
&<1+\frac{np(p-2)}{[n(p-2)^2+p]\lm M^{p-2}}(-t)^{n(p-2)/\lm} h(t)^{\al(p-2)}.
\end{aligned}
\end{equation}
Notice that $j$ is still completely undetermined. In order to prove \ref{cond-first} in Definition~\ref{def-barrier}, we need to show that the domain defined by \eqref{Eq:petr:6} contains $\Theta$. Indeed, by \eqref{Eq:petr:1}, 
we have 
\[
\biggl(\frac{|x|}{(-t)^{1/\lm}}\biggr)^{p/(p-1)}< KM^{p-2}
    \quad \text{in } \Theta.
\]
Taking into account that 
$ (1+s)^{(p-1)/(p-2)}\le1+\frac{p-1}{p-2}s(1+s)^{1/(p-2)}$ for $s \ge 0$
yields
\begin{align*}
&\biggl[1+\frac{p-2}{pj\lm^{1/(p-1)}}\biggl(\frac{|x|}{(-t)^{1/\lm}}
   \biggr)^{p/(p-1)}\biggr]^{(p-1)/(p-2)}\\
&\kern 5em \le 1+\frac{p-1}{pj\lm^{1/(p-1)}}\biggl(\frac{|x|}{(-t)^{1/\lm}}
   \biggr)^{p/(p-1)}
\biggl(1+\frac{p-2}{pj\lm^{1/(p-1)}}KM^{p-2}\biggr)^{1/(p-2)}\\
&\kern 5em \le 1+\frac{L}{j}\biggl(\frac{|x|}{(-t)^{1/\lm}}\biggr)^{p/(p-1)}\\
&\kern 5em <1+\frac Lj K(-t)^{n(p-2)/\lm}h(t)^{\al(p-2)},
\end{align*}
where $L$ is independent of $j  \ge 1$.
Since both $K$ and $M$ are independent of $j$, provided we choose $j$ large enough, we can always conclude that
\[
\frac Lj K<\frac{np(p-2)}{[n(p-2)^2+p]\lm M^{p-2}},
\]
therefore satisfying \eqref{Eq:petr:6}. 

Finally, we show that \ref{cond-second-weak}
in Definition~\ref{def-barrier} is satisfied. 
Let $y(t)$, $-1/2e <t<0$,  be the solution to
\[ 
  \biggl(\frac{y}{(-t)^{1/\lambda}}\biggr)^{p/(p-1)}
   =K(-t)^{n(p-2)/\lambda}h(t)^{\al(p-2)}.
\]
Then for $(x,t) \in \bdy \Theta$ with $|x|=y(t)$, we have 
\begin{align*}
w(x,t)=&-\eps h(t)^{\al}
\biggl(j+\frac{p-2}{p\lm^{1/(p-1)}}K(-t)^{n(p-2)/\lm}
h(t)^{\al(p-2)}\biggr)^{(p-1)/(p-2)}\\
&+\eps j^{(p-1)/(p-2)}h(t)^\al+A(-t)^{n(p-2)/\lm}h(t)^{\al(p-1)}\\
=&-\eps j^{(p-1)/(p-2)}h(t)^{\al}
\biggl(1+\frac{K(p-2)}{pj\lm^{1/(p-1)}}(-t)^{n(p-2)/\lm}
h(t)^{\al(p-2)}\biggr)^{(p-1)/(p-2)}\\
&+\eps j^{(p-1)/(p-2)}h(t)^\al+A(-t)^{n(p-2)/\lm}h(t)^{\al(p-1)}.
\end{align*}
Using that $(1+s)^{(p-1)/(p-2)}\le 1 +B s$ for $0\le s \le \ga$ and
that $B$ only depends on the bound $\ga$,
we have with $\Kt=KB$,
\begin{align*}
w(x,t)\ge&-\eps j^{(p-1)/(p-2)}h(t)^{\al}
\biggl(1+\frac{\Kt(p-2)}{pj\lm^{1/(p-1)}}(-t)^{n(p-2)/\lm}
h(t)^{\al(p-2)}\biggr)\\
&+\eps j^{(p-1)/(p-2)}h(t)^\al+A(-t)^{n(p-2)/\lm}h(t)^{\al(p-1)}\\
=&-\eps j^{(p-1)/(p-2)}\frac{\Kt(p-2)}{pj\lm^{1/(p-1)}}(-t)^{n(p-2)/\lm}h(t)^{\al(p-1)}\\
&+A(-t)^{n(p-2)/\lm}h(t)^{\al(p-1)}\\
=&\eps j^{(p-1)/(p-2)}\biggl(\frac{n(p-2)\eps^{p-2}}{n(p-2)^2+p}
  -\frac{\Kt(p-2)}{pj\lm^{1/(p-1)}}\biggr)(-t)^{n(p-2)/\lm}h(t)^{\al(p-1)}.
\end{align*}
For $(x,t) \in \bdy \Theta$ with $t=-\frac{1}{2e}$ and 
$|x|<y(t)$, we instead have 
\[
    w(x,t) \ge w(x',t),
\]
where $|x'|=y(t)$, since $f$ is a negative function.
Together with the last estimate, this shows that
$\{w_j\}_{j=k}^\infty$ is a  barrier family,
provided $k$ is large enough.
\end{proof}

%%%%%%%%%%%%%%%%%%%%%%%%%%%%%%%%%%%%%%%%%%%%%%%
\section{Remarks on the regularity of a final point for 
\texorpdfstring{$1<p<2$}{}}\label{S:final}

Due to the analogies in the definition of the Barenblatt fundamental solution for
$p>2$ and $\frac{2n}{n+1}<p<2$, one would expect that in the singular 
supercritical
range $\frac{2n}{n+1}<p<2$
the origin $(0,0)$ is regular with respect to the domain
\begin{equation}\label{Eq:petr:1sing}
\begin{aligned}
\Theta=&\biggl\{(x,t)\in\Rno: -\frac1{2e}<t<0 \text{ and}\\ 
&\biggl(\frac{|x|}{(-t)^{1/\lambda}}\biggr)^{p/(p-1)}
   <K(-t)^{n(2-p)/\lambda}\efoftsing^{\al(2-p)}\biggr\},
\end{aligned}
\end{equation}
where $K$ and $\alpha$ denote arbitrary positive constants, and $\lm=n(p-2)+p$. 
Notice that the above-indicated range of $p$ ensures that $\lm>0$.

With methods that are very similar to the ones used in the previous section, it is not hard
to prove that for proper values of $\eps$ and for sufficiently large values of $j$, the function
\begin{align*}
w=&-\eps\efoftsing^\al\biggl[j-\frac{2-p}{p\lm^{1/(p-1)}}
     \biggl(\frac{|x|}{(-t)^{1/\lm}}\biggr)^{p/(p-1)}\biggr]^{-(p-1)/(2-p)}\\
     &+\eps j^{-(p-1)/(2-p)}\efoftsing^\al\\
     &+\frac{n(2-p)\eps j^{-(p-1)/(2-p)}}{\lm M^{2-p}}
     (-t)^{n(2-p)/\lm}\efoftsing^{\al(3-p)}
\end{align*}
is \emph{one} barrier. 
Unfortunately, unlike the $p>2$ case, now as $j\to\infty$, $w\to0$ and 
condition~\ref{cond-second-weak} of Definition~\ref{def-barrier} cannot be satisfied. 
Therefore, we do not have a whole \emph{family} of barriers, and we cannot conclude 
regularity.

What we can prove is a somewhat weaker result, valid in the whole singular 
range $1<p<2$. 

\begin{prop}\label{Prop:8:1}
Let $1<p<2$. The origin $(0,0)$ is regular with respect to the domain
\begin{equation}\label{Eq:petr:1:sing}
\Theta=\{(x,t)\in\Rno: -1 < t<0 \text{ and }  |x|^l<K(-t)\}
\end{equation}
if $0 <l < p$  and $K>1$ is arbitrary. 
\end{prop}

Besides being nonoptimal, the result of Proposition~\ref{Prop:8:1} is also not stable, and it cannot be: Indeed, the barrier $u_j$ below is defined only for $1<p<2$ and for $p=2$ the regularity of the origin as an end point is characterised by the original 
Petrovski\u\i\ criterion.

\begin{proof}
Let 
\[
G^j=\biggl\{x:|x|<\frac12\sqrt{\frac{2-p}j}\biggr\}\times(-\infty,0),
\]
and consider the function $u_j:G^j\to\R$ defined by
\[
u_j(x,t)=j^\al\biggl(-\frac tj\biggr)^{1/(2-p)}\biggl(\frac{2-p}j-|x|^2\biggr),
\]
where $\al$ is a positive parameter to be determined and $j>1$ is arbitrary.
First of all, notice that
\[
\lim_{G^j\ni(x,t)\to(0,0)}u_j(x,t)=0.
\]
Next, we want to show that $u_j$ is \p-superparabolic in $G^j$, provided $\al$ is properly chosen.
We have
\begin{align*}
{\partial_t u_j} &=-\frac{j^{\al-1}}{2-p}\biggl(-\frac tj\biggr)^{(p-1)/(2-p)}
     \biggl(\frac{2-p}j-|x|^2\biggr),\\
\Delta_p u_j &=-2^{p-1}j^{\al(p-1)}\biggl(-\frac tj\biggr)^{(p-1)/(2-p)}(n+p-2)|x|^{p-2}.
\end{align*}
Therefore in $G^j$,
\begin{align*}
& \kern -2em {\partial_t u_j}-\Delta_p u_j \\
&=\biggl(-\frac tj\biggr)^{(p-1)/(2-p)}
    \biggl[2^{p-1}j^{\al(p-1)}(n+p-2)|x|^{p-2}-\frac{j^{\al-1}}{2-p}
         \biggl(\frac{2-p}j-|x|^2\biggr)\biggr] \kern -2em\\
&\ge\biggl(-\frac tj\biggr)^{(p-1)/(2-p)}
   \biggl[2^{p-1}j^{\al(p-1)}(n-1)\biggl(\frac1{|x|}\biggr)^{2-p}-j^{\al-2}\biggr]\\
 &\ge\biggl(-\frac tj\biggr)^{(p-1)/(2-p)}
   \biggl[\frac2{(2-p)^{1-p/2}}(n-1)j^{\al(p-1)+1-p/2}-j^{\al-2}\biggr]\\
&\ge\biggl(-\frac tj\biggr)^{(p-1)/(2-p)}(j^{\al(p-1)+1-p/2}-j^{\al-2}),
\end{align*}
and, since $j>1$, we have shown that $u_j$ is \p-superparabolic if $\al>0$ 
is such that 
\[
\al(p-1)+\frac{2-p}2>\al-2,\quad\text{or equivalently}\quad
\al<\frac2{2-p}+\frac12.
\]

Finally, we show that condition~\ref{cond-second-weak}
 in Definition~\ref{def-barrier} is satisfied, provided $\al$ is properly chosen. 
Let 
\[
    m_j =\inf_{(x,t) \in \Theta \cap \bdy G^j} u_j(x,t) >0.
\]
Since for any $(x,t)\in G^j$ we have
\[
u_j(x,t)\ge\frac{3(2-p)}4j^{\al-1}\biggl(-\frac tj\biggr)^{1/(2-p)}
          =\frac{3(2-p)}4 j^{\al-1-1/(2-p)}(-t)^{1/(2-p)},
\]
we see that  $m_j=u_j(x,t)$ if $\abs{x}=\frac12 \sqrt{(2-p)/j}$ 
and  $t=-|x|^l/K$, i.e.
\[
    m_j = \frac{3(2-p)}4 
      \biggl( \frac{\bigl(\tfrac{1}{2}\sqrt{2-p}\bigr)^l}{K}\biggr)^{1/(2-p)}
           j^{\al-1-(1+l/2)/(2-p)}.
\]
Since 
also $u_j(x,t) \ge m_j$ when $(x,t) \in \bdy (\Theta \cap G^j)$ and $t=-1$, 
it follows from the pasting lemma (Lemma~\ref{lem-pasting}) that
\[
   w_j=\begin{cases}
            \min\{u_j,m_j\} & \text{in } G^j \cap \Theta, \\
            m_j & \text{in } \Theta \setm G^j,
            \end{cases}
\]
is a barrier family, 
provided we choose 
$\alp$ so that 
\[
1+\frac{1+ \tfrac{1}{2}l}{2-p}<\al<\frac2{2-p}+\frac12,
\]
which is possible as  $0<l<p$.
\end{proof}

%%%%%%%%%%%%%%%%%%%%%%%%%%%%%%%%%%%%%%%
\section{Open problems}\label{S:Open}
Here we collect a short list of open problems.
\begin{enumerate}
\renewcommand{\theenumi}{\textup{\arabic{enumi}.}}%
\item In Theorem~\ref{thm:barrier-char} we characterise the regularity of a 
boundary point in terms of the existence of a \emph{family} of barriers. 
Is it really necessary to have a family of barriers? 
Or, stated otherwise,
does there exist an open set $\Theta$ and an irregular boundary
point $\xi_0 \in \bdy \Theta$ such that there is \emph{one} barrier at $\xi_0$?
We think that this 
question has a positive answer, but unfortunately, 
we could not find such an example. 
\item The previous question is directly linked
to the stability as $p\to2$ of the results given in 
Sections~\ref{S:petr} and~\ref{S:final}: if a family is indeed 
necessary, then, in order to 
prove stability, one should be able to find a family of barriers, which in the limit for $p=2$ converges
to a set of barriers, all multiples of a \emph{minimal} one; this has to be the case, 
because 
a \emph{single} barrier suffices when $p=2$.
\item When $p>2$, can we stabilise the estimates of Proposition~\ref{Prop:Petr-deg} and recover the classical Petrovski\u\i\ condition?
\item In the range $1<p<2$ (or at least $\frac{2n}{n+2}<p<2$) can we improve the result of Proposition~\ref{Prop:8:1} and build a family of barriers, which in the limit as $p\to2^\limminus$ yields the
original Petrovski\u\i\ condition? 
\item The problem of the north pole for the exterior ball
condition remains open,
when $n>1$. 
\item As we mentioned in Section~\ref{sect-gen-cone}, it would be interesting to see, 
whether the so-called
tusk condition holds for a general $p>1$, and not just for $p=2$. 
\item 
A related interesting problem is the \emph{resolutivity problem}:
If $f \in C(\Theta)$ is then always $\uP f = \lP f$?

\end{enumerate}
%%%%%%%%%%%%%%%%%%%%%%%%%%%%%%%%%%%%%%%

\end{document}